\documentclass{article}
\usepackage[utf8]{inputenc}
\usepackage[english]{babel}
\usepackage{amsmath}
\usepackage{amsfonts}
\usepackage{mathrsfs}
\usepackage{amssymb}
\usepackage{dsfont}
\usepackage{amsthm}
\usepackage{hyperref,polynom}
\usepackage{xfrac}
\usepackage{defs}
\usepackage{fonts}
\usepackage{esint}
\usepackage{cite}
\usepackage{wrapfig} 
\usepackage[export]{adjustbox} 
\usepackage{todonotes}
\usepackage{mathtools} 
\usepackage{graphicx} 
\usepackage{array} 
\usepackage{enumerate} 
\usepackage{hyperref} 
\usepackage{setspace} 
\usepackage{lineno} 
\usepackage[margin=1in]{geometry}
\usepackage{setspace}
\newtheorem{theorem}{Theorem}[section]
\newtheorem{proposition}[theorem]{Proposition}
\theoremstyle{definition}

\theoremstyle{plain}
\newtheorem{lemma}[theorem]{Lemma}
\newtheorem{corollary}{Corollary}[theorem]
\theoremstyle{remark}
\newtheorem{remark}[theorem]{Remark}

\title{A Fractional Korn-type inequality for smooth domains
{and a regularity estimate for nonlinear nonlocal systems of equations} \thanks{Support from NSF  DMS-1615726 is gratefully acknowledged.}}

\author
{Tadele Mengesha and James M. Scott}
\makeatletter
\def\namedlabel#1#2{\begingroup
    #2%
    \def\@currentlabel{#2}%
    \phantomsection\label{#1}\endgroup
}
\makeatother

\newcommand{\vphi}{\varphi}

\DeclareMathOperator*{\suppp}{supp}
\DeclareMathOperator*{\distt}{dist}
\newcommand{\intdm}[3]{\displaystyle \int_{#1} #2 \, \mathrm{d}#3}
\newcommand{\iintdm}[5]{\int_{#1} \int_{#2}  #3 \, \mathrm{d}#4 \, \mathrm{d}#5}

\newcommand{\intdmt}[4]{ \int_{#1}^{#2} #3 \, \mathrm{d}#4}
\newcommand{\iintdmt}[6]{ \int_{#1}^{#2} \int_{#3} #4 \, \mathrm{d}#5 \, \mathrm{d}#6}

\newcommand{\EquationReference}[2]{\mathrel{\overset{\makebox[0pt]{\mbox{\normalfont\tiny\sffamily #1}}}{#2}}}

\newcommand{\diffqbunorm}{\left| \big( \bu(\bx)-\bu(\by) \big) \cdot \frac{\bx-\by}{|\bx-\by|}\right|}

\newcommand{\dotbxy}{\frac{\bx-\by}{|\bx-\by|}}

\everymath{\displaystyle}

\begin{document}

\maketitle


\numberwithin{equation}{section}


\begin{abstract}
In this paper we prove a fractional analogue of the classical Korn's first inequality. The inequality makes it possible to show the equivalence of a function space of vector field characterized by a Gagliardo-type seminorm with {\em projected difference} with that of a corresponding fractional Sobolev space.   As an application, we will use it to obtain a Caccioppoli-type inequality for a nonlinear system of nonlocal equations, which in turn is a key ingredient in applying known results to prove a higher fractional differentiability result for weak solutions of  the nonlinear system of nonlocal equations. The regularity result we prove will demonstrate that a well-known self-improving property of scalar nonlocal equations will hold for strongly coupled systems of nonlocal equations as well.  
\end{abstract}

\section{Introduction and statement of main results}
For $d\geq 2$, suppose that $\Omega\subset\mathbb{R}^{d}$ is a bounded domain with $C^{1}$ boundary. For $s\in (0, 1)$ and $1<p<\infty$, define the space $\cX^{s}_{p}(\Omega)$ to be the closure of $[C_{c}^{1}(\Omega)]^{d}$ with respect to the norm $\|\cdot\|_{X^{s,p}}$ given by 
\[
\|\bu\|_{\cX^{s,p}}=  [\bu]_{\cX^{s,p}(\Omega)} + \Vnorm{\bu}_{L^p(\Omega)},\quad 
\]
where the seminorm $[\bu]_{\cX^{s,p}(\Omega)}$ is given by 
$
[\bu]_{\cX^{s,p}(\Omega)}^{p} := \iintdm{\Omega}{\Omega}{\frac{\diffqbunorm^p}{|\bx-\by|^{d+sp}}}{\by}{\bx}. 
$
The space $[C_{c}^{1}(\Omega)]^d$ denotes the set of continuously differentiable vector fields $\bu: \Omega\to \mathbb{R}^{d}$ whose support is compactly contained in $\Omega$. The seminorm $[\bu]_{\cX^{s,p}(\Omega)}^{p}$,  which is based on the size of the {\em projected difference} $\diffqbunorm$, is smaller  than the well-known Aronszajn-Slobodeckij-Gagliardo seminorm 
$
|\bu|_{W^{s,p}(\Omega)}^{p} = \iintdm{\Omega}{\Omega}{\frac{|\bu(\bx)-\bu(\by)|^p}{|\bx-\by|^{d+sp}}}{\by}{\bx} 
$
that uses the difference $|\bu(\bx)-\bu(\by)|$.  Each of these seminorms measure somewhat different things. Intuitively this can be seen from the simple Taylor's expansion that  for a given  smooth vector field $\bu$, while the difference $|\bu(\bx)-\bu(\by)|= |\nabla \bu (\bx) (\by-\bx)| + O(|\by-\bx|)$, the projected difference $\diffqbunorm = \left|\text{Sym}(\nabla \bu (\bx)) (\by-\bx)\cdot {\by-\bx\over |\by-\bx|}\right| + O(|\by-\bx|)$, where $\text{Sym}(\nabla \bu (\bx)) $ is the symmetric part of the gradient matrix defined as ${1\over 2}(\nabla \bu(\bx) + \bu(\bx)^{T})$.   

In this paper we establish connection between these seminorms $|\cdot|_{\cX^{s,p}(\Omega)}$ and  $|\cdot|_{W^{s,p}(\Omega)}$. In fact, motivated by the classical Korn's inequality which establishes the equivalence of  the seminorms $\|\nabla \bu\|_{L^{p}}$ and $\| \text{Sym}(\nabla \bu)\|_{L^{p}}$ for compactly supported vector fields, see \cite{CiarletPhilippeG2010OKi, doi:10.1137/1037123} for review, it is reasonable to ask whether this equivalence is true for $[\bu]_{\cX^{s,p}}$ and $|\bu|_{W^{s,p}(\Omega)}$. In the event $\Omega=\mathbb{R}^{d}$, this question was answered in the affirmative by the authors in \cite{MengeshaScott2018Korn}. For $p=2$ and $\Omega$ the half-space $\mathbb{R}^{d}_{+}$, an affirmative answer was  given earlier in \cite{MengeshaTadele2019FKaH}. Continuing that effort we prove in this paper that  the space $\cX^{s,p}(\Omega)$ is precisely $[W^{s, p}_0(\Omega)]^d$ with equivalent norms for sufficiently smooth domains. The function space $[W^{s, p}_0(\Omega)]^d$  is the closure of $[C_{c}^{1}(\Omega)]^d$ with respect the the norm $\|\cdot\|_{W^{s, p}}=|\bu|_{W^{s,p}(\Omega)} + \Vnorm{\bu}_{L^p(\Omega)}$.The main result is the following:

\begin{theorem}[Fractional Korn's Inequality for Bounded $C^1$ Domains]\label{thm:KornsC1}
Let $d \geq 2$, $s \in (0,1)$, $p \in (1,\infty)$ with $sp \neq 1$. Let $\Omega \subset \bbR^d$ be a bounded domain with $C^1$ boundary $\p \Omega$. Then there exists a constant $C=C(d,s,p,\Omega)$ such that for every $\bu \in \big[ C^1_c(\Omega) \big]^d$ 
\begin{equation}\label{eq:KornsInequalityC1}
|\bu|_{W^{s,p}(\Omega)} \leq C \left( [\bu]_{\cX^{s,p}(\Omega)} + \Vnorm{\bu}_{L^p(\Omega)} \right)\,.
\end{equation}
By density the inequality holds for all $\bu \in [W^{s,p}_{0}(\Omega)]^d$.
\end{theorem}

We emphasize that this work focuses on vector fields that vanish on the boundary of the domain. As such the fractional Korn's inequality stated in the above theorem can be thought of as a fractional analogue to the classical Korn's first inequality.  The more interesting question of whether $[W^{s,p}(\Omega)]^d=\{\bu\in [L^{p}(\Omega)]^d: |\bu|_{W^{s,p}} < \infty\}$ is equal to the space $\{\bu\in [L^{p}(\Omega)]^d: [\bu]_{\cX^{s,p}} < \infty\}$ is unanswered here. We believe that a properly  quantified notion of trace on boundary for vector fields in  $\{\bu\in [L^{p}(\Omega)]^d: [\bu]_{\cX^{s,p}} < \infty\}$, which we lack now, is the first step in showing the equality of the spaces. 
We also do not make any remark on sufficient conditions on a radial kernel $\rho$ and the domain $\Omega$ such that the general function space $\cS^{p}_\rho(\Omega) =\{\bu\in [L^{p}(\Omega)]^d: [\bu]_{\cS^{p}_\rho} < \infty\}$ where $[\bu]_{\cS^{p}_\rho}^{p}= \int_\Omega \int_\Omega \rho(\by-\bx)\diffqbunorm^{p} \, \rmd \by \, \rmd \bx$
 is equal to the space of vector fields $W^{p}_{\rho}(\Omega)=\{\bu\in [L^{p}(\Omega)]^d: \int_\Omega \int_\Omega \rho(\by-\bx)|{\bf u}(\by) - {\bf u}(\bx)|^{p} \, \rmd \by \, \rmd \bx < \infty\}$.  If $\rho$ is locally integrable, it is known that each of these spaces coincide with $[L^{p}(\Omega)]^{d}$ \cite{mengesha2015VariationalLimit}.  However, for non-integrable kernels the spaces are proper subsets of $[L^{p}(\Omega)]^{d}$. In fact, under extra assumptions that insure  singularity of the kernel $\rho$, the compact embedding of the spaces $\cS^{p}_\rho(\Omega)$ and $W^{p}_\rho(\Omega)$ in  $[L^{p}(\Omega)]^{d}$ is proved in \cite{Du-Mengesha-Tian-arxiv}  and \cite{BBM, Ponce2004}  respectively. 
We  note that the spaces $\cX^{s,p}(\Omega)$ and $[W^{s,p}_0(\Omega)]^d$ are special subspaces that correspond to the fractional kernel $\rho(|{\bz}|) = |{\bz}|^{-(d+ps)}.$

The proof of the theorem follows standard procedures where we first prove the same result for epigraphs and use  a partition of unity to localize near the boundary of the domain $\partial \Omega$. To that end, let $f : \bbR^{d-1} \to \bbR$ be a $C^{1}$ globally Lipschitz function, with $f(0)=0$ and $\grad f(0) = 0$. 
We say that an open subset $D$ of $\bbR^{d}$ is an \textit{epigraph} supported by $f$ if 
\[
D=\{ (\bx',x_d) \in \bbR^d \, : \, x_d > f(\bx') \}.
\]
The boundary $\partial D$ of the epigraph $D$ is precisely the graph of the function $x_d=f(\bx')$. For a given globally Lipschitz function $f$ as above we denote its Lipschitz constant by $\rmM := \Vnorm{\grad f}_{L^{\infty}(\bbR^{d-1})}$. 
\begin{theorem}[Fractional Korn's Inequality for epigraphs]\label{thm:KornsGraph}
Let $d \geq 1$, $s \in (0,1)$, $p \in (1,\infty)$ with $sp \neq 1$. Then there exists a constant $\rmM_0 > 0$ depending only on $d$, $s$, and $p$ with the following property: for any epigraph determined by $f$ with Lipschitz constant $\rmM < \rmM_0$, there exists a constant $C = C(d,s,p,\rmM_0)$ such that for every $\bu \in \big[ C^1_c(D) \big]^d$, 
\begin{equation*}
|\bu|_{W^{s,p}(D)} \leq C [\bu]_{\cX^{s,p}(D)}\,.
\end{equation*}
\end{theorem}
We will prove the fractional Korn’s inequality for an epigraph $D$
by first extending the vector fields in $\cX^{s,p}(D)$ to be 
defined in the whole space in such a way that the extended functions belong to
$\cX^{s,p}(\bbR^{d})$ and their  seminorm is controlled by the norm on the epigraph.  Once we establish that we can then apply the fractional Korn’s inequality for vector fields defined on all of $\bbR^d$ that is proved in \cite{MengeshaScott2018Korn}. Note that extending functions with proper control of their norms is a nontrivial task as the commonly used reflection across
the boundary of $D$ would not be preserving the seminorm $[\cdot]_{s,p}$. Nor would extending by zero be appropriate, since it is
not clear how to control the norm of the extended function. We instead use an
extension operator that has been used by Nitsche in \cite{nitsche1981korn} in his simple proof of
Korn’s second inequality along with the fractional Hardy-type inequality proved in \cite{MengeshaTadele2019FKaH} to show the boundedness of the
extension operator with respect to the seminorm $[\cdot]_{\cX^{s,p}}$. The precise statement is stated as follows. 
\begin{proposition}\label{ext-operator}
There exists a universal constant $\rmM_0 > 0$  such that 
for any epigraph $D$ supported by $f$ whose Lipschitz constant $M<\rmM_0$, there exists a bounded  extension operator
\begin{equation*}
\mathrm{E} : \big[ C^1_c(D) \big]^d \to \big[ C_c^{1}(\bbR^d) \big]^d
\end{equation*}
with the property that for every $\bu \in \big[ C^1_c(D) \big]^d$, 
\begin{equation}\label{eq:ExtensionBound}
\Vnorm{\mathrm{E}(\bu)}_{\cX^{s,p}(\bbR^d)} \leq C\left( \Vnorm{\bu}_{\cX^{s,p}(D)} + \rmM\Vnorm{\bu}_{W^{s,p}(D)} \right)\,,
\end{equation}
where the constant $C$ depends only on $\rmM_0$, $d$, $s$, and $p.$ 
\end{proposition}
Notice that the proposition applied to the half-space $D=\mathbb{R}^{d}_{+}$ which corresponds to $f=0,$ and so $M=0$ is precisely the extension operator proved in \cite{MengeshaTadele2019FKaH}.  
The proposition can therefore be viewed as a generalization of \cite[Theorem 2.2]{MengeshaTadele2019FKaH} for general epigraphs. The presence of the Lipschitz constant $M$ in the left hand side of \eqref{eq:ExtensionBound} as a multiplier of $\Vnorm{\bu}_{W^{s,p}(D)} $ is crucial in what we do next as it will enable us to absorb this term on the right-hand side.

As an application of Korn's inequality, we study  the higher fractional differentiability and higher integrability of weak solutions of the nonlinear system of nonlocal  equations 
\begin{equation}\label{eq:PLaplaceSystem}
p.v.\intdm{\bbR^d}{\frac{A(\bx,\by)}{|\bx-\by|^{d+sp}}|\cD(\bu)(\bx,\by)|^{p-2}\cD(\bu)(\bx,\by)}{\by}= \bff(\bx)\,, \qquad \bx \in \bbR^d\,,
\end{equation}
where $d \geq 2$, $p\geq 2$, $0<s<1$, 
the quantity $\cD(\bu)(\bx,\by)$ denotes the projected difference given by $
\cD(\bu)(\bx,\by) := \big( \bu(\bx)-\bu(\by) \big) \cdot \frac{\bx-\by}{|\bx-\by|}$. 
The function 
$A : \bbR^d \times \bbR^d \to (0,\infty)$ serves as a coefficient and is measurable, symmetric ($A(\bx,\by) = A(\by,\bx)$) and satisfies the ellipticity condition
\begin{equation}\label{eq:ConditionsOnA}
0 < \frac{1}{\Lambda} \leq A(\bx,\by) \leq \Lambda\,, \qquad \bx\,, \by \in \bbR^d\,.
\end{equation}
The system of equations \eqref{eq:PLaplaceSystem} is strongly coupled and for $p=2$, the equation appears in linearized peridyanmics, a nonlocal model of continuum mechanics \cite{Silling2000, Silling2007, Silling2010}, corresponding to a singular fractional kernel. 

Given $\bff \in L^{1}_{loc}(\mathbb{R}^{d})$, by a weak solution of \eqref{eq:PLaplaceSystem} we mean $\bu \in [W^{s,p}(\bbR^d)]^{d}$ such that 
\begin{equation}\label{eq:PLaplaceWeakSoln}
\cE_{p,A}(\bu,\vphi)= \int_{\bbR^{d}} \bff(\bx)\vphi(\bx)d\bx\,,\quad \text{ $\forall \vphi \in \big[ C^{\infty}_c(\bbR^d) \big]^d$  }
\end{equation}
where the integral form $\cE$ is given by  
\[\cE_{p,A}(\bu,\vphi)=\iintdm{\bbR^d}{\bbR^d}{\frac{A(\bx,\by)}{|\bx-\by|^{d+sp}} |\cD(\bu)(\bx,\by)|^{p-2} \cD(\bu)(\bx,\by) \, \cD(\vphi)(\bx,\by)}{\by}{\bx}. \]
Existence of solution satisfying \eqref{eq:PLaplaceWeakSoln} can be proved via variational methods, say under some complementary conditions on ${\bf u}$ outside of a bounded set $\Omega$. For example, for $\bff\in L^{p'_\ast}(\bbR^{d})$, where $(p')_\ast = {p'd\over d+p's}$, and ${1\over p} + {1\over p'}=1$, we can minimize the energy 
\[
{\bf u}\mapsto \cE_{p,A}(\bu,\bu)-\int_{\bbR^{d}} \bff(\bx)\bu(\bx) \, \rmd \bx
\]
over the subspace $\{{\bf u}\in [W^{s,p}(\bbR^d)]^{d}: {\bf u}=0\quad \text{on $\Omega$} \}$. Notice that the energy space associated to the above variational problem is precisely $\cX^{s,p}(\bbR^d)$, and so by the fractional Korn's inequality is equal to $[W^{s,p}(\bbR^d)]^{d}$. Coercivity of the above functional energy can be proved via the Poincar\'e-Korn inequality \cite{Du-Navier1, MengeshaDuElasticity} and again by the fractional Korn's inequality. 

Our focus here is on the self-improving properties of the nonlinear system of nonlocal  equations. By ``self-improving" we mean the increase in higher fractional  differentiability and integrability of solutions to nonlocal equations by  virtue of being a solution to the nonlocal system corresponding to ${\bf f}$ that has improved integrability.  To be precise we have the following:
\begin{theorem}\label{thm:MainRegularityResult}
Suppose that $p \in [2,\infty)$ and $s \in (0,1)$ with $sp < n$ and $sp \neq 1$. Let $\delta_0 > 0$ be given, and for $\delta \in (0,\delta_0)$ assume that  $\bff \in \big[ L^{p'_\ast + \delta}(\bbR^d) \big]^d$. Suppose that  $A$ satisfies \eqref{eq:ConditionsOnA} and that $\bu \in \big[ W^{s,p}(\bbR^d) \big]^d$ is a weak solution to \eqref{eq:PLaplaceSystem} satisfying \eqref{eq:PLaplaceWeakSoln}. Then there exists $\veps_0 \in (0,1-s)$ depending on $d$, $s$, $p$, $\delta$ and $\Lambda$ such that $\bu \in \big[W^{s+\veps_0, p+\veps_0}_{loc}(\bbR^d) \big]^d$.
\end{theorem}

For scalar equations, such self-improving properties have been proved by \cite{kuusi2015, Auscher2018nonlocal}, where it was explained that this property is unique to solutions of nonlocal equations.  The result stated in the above theorem confirms that such properties also extend to strongly coupled systems of nonlocal equations such as\eqref{eq:PLaplaceSystem}.

The paper is organized as follows. In Section \ref{sec:epigraph} we prove the fractional Korn's inequality for epigraphs, Theorem \ref{thm:KornsGraph}. This is also the section where the theorem on the extension operator, Theorem \ref{ext-operator} will also be proved.  In Section \ref{sec:smoothdomains}, we prove the main result, Theorem \ref{thm:KornsC1}. In the last section, the proof of the self-improving property of the coupled system \eqref{eq:PLaplaceSystem} will be discussed.

\section{Fractional Korn's inequality for epigraphs}\label{sec:epigraph}
In this section, we will prove Theorem \ref{thm:KornsGraph}. As we indicated earlier, the main tool is Proposition \ref{ext-operator} which states the existence of a bounded extension operator for vector fields defined over epigraphs. Thus, the main task is proving Proposition \ref{ext-operator} which we do so as follows. As before, we assume that $D$ is an epigraph supported by $f$. We introduce $D_{-} = \bbR^{d}\setminus \overline{D}$ which can be expressed in terms of the defining function as $D_{-} = \{(\bx', x_d)\in \bbR^{d}: x_d<f(\bx')\}$ The following supporting lemma shows that $D$ and $D_{-}$ are diffeomorphic and its proof follows from direct calculation. 

\begin{lemma}\label{lma-PropertiesOfCoordinateChange}
For $\eta > 0$, define $\Phi_{\eta} : D_- \to D$ by
$\Phi_{\eta}(\bx) := \big( \bx',f(\bx') + \eta(f(\bx') - x_d) \big)\,.$
Then $\Phi_{\eta}$ is a $C^1$ diffeomorphism, with inverse
$\left( \Phi_{\eta} \right)^{-1}(\bx) := \left( \bx',f(\bx') + \frac{1}{\eta}(f(\bx')-x_d) \right)\,,$
and
$\det \grad \Phi_{\eta} = -\eta\,.$
\end{lemma}
We remark for the diffeomorphism $\Phi$ in the above lemma, we can compute 
$\grad \Phi_\eta = \begin{bmatrix}
\bbI^{d-1}&(1+\eta)\grad f\\
0&-\eta
\end{bmatrix}$ where $\bbI^{d-1}$ is the identity matrix in $\bbR^{d-1 \times d-1}$. Similarly, $\grad (\Phi_\eta)^{-1} = \begin{bmatrix}
\bbI^{d-1}&(1+ {1\over \eta})\grad f\\
0&-{1\over \eta}
\end{bmatrix}$. As a consequence, direct calculations show that $\Vnorm{\grad \Phi_{\eta}}_{L^{\infty}(D_-)}$ and $\Vnorm{\grad (\Phi_{\eta})^{-1}}_{L^{\infty}(D)}$  are given by $\sqrt{d-1 + \eta^{2} + (1+\eta)^{2}M}$ and $\sqrt{d-1 + {1\over \eta^{2}} + (1+{1\over \eta})^{2}M}$ respectively. 
Specifically, both $\Vnorm{\grad \Phi_{\eta}}_{L^{\infty}(D_-)}$ and $\Vnorm{\grad (\Phi_{\eta})^{-1}}_{L^{\infty}(D)}$ are bounded from below by a constant that depends only on $d$ and $\eta$ uniformly in $M$. 

The diffeomophism $\Phi_\eta$ also satisfies  the following geometric inequality which says that if the Lipschitz constant of the supporting function of $D$  is sufficiently small then the distances from $\bz$ and $(\Phi_{\eta})^{-1}(\bz)$ to any arbitrary point $\by$ in $D$ are comparable. 
\begin{lemma}\label{lma:GeometricInequality}
Let $\eta > 0$ and let $C_{\eta}$ be a constant such that $C_{\eta} > \max \{ 1,\eta \}$. Then if 
$\textstyle \mathrm{M}^2 \leq \frac{\left( C_{\eta}^2-\eta^2 \right) \left( C_{\eta}^2-1 \right) }{(C_{\eta}^2 + \eta)^2}\,,$
then for every $\bz$, $\by \in D$.
\begin{equation}\label{eq:ComparisonOfMeasures}
|\bz - \by| \leq C_{\eta} \left| (\Phi_{\eta})^{-1}(\bz) - \by \right|
\end{equation}

\end{lemma}

\begin{proof}
Let $\alpha= z_d - f(\bz')$, $\beta = f(\bz') - f(\by')$, $\gamma = y_d - f(\by')$, and $\delta = |\bz'-\by'|$. In order to show \eqref{eq:ComparisonOfMeasures} it suffices to show that
\begin{equation*}
\delta^2 + |\alpha+\beta-\gamma|^2 \leq C_{\eta}^2 \delta^2 + C_{\eta}^2 \left| - \frac{1}{\eta} \alpha + \beta - \gamma \right|^2\,,
\end{equation*}
i.e.
\begin{equation}\label{eq:NonnegPolynomial1}
0 \leq \left( \frac{C_{\eta}^2}{\eta^2} - 1 \right) \alpha^2 -2 \left( \frac{C_{\eta}^2}{\eta} + 1 \right) \alpha\beta + 2 \left( \frac{ C_{\eta}^2}{\eta} + 1 \right) \alpha \gamma+ (C_{\eta}^2 - 1) (\beta-\gamma)^{2} + (C_{\eta}^2 - 1) \delta^2\,.
\end{equation}
The term $\textstyle (C_{\eta}^2 - 1) (\beta-\gamma)^2$ is nonnegative by assumption on $C_{\eta}$. The term $\textstyle 2 \left( \frac{ C_{\eta}^2}{\eta} + 1 \right) \alpha\gamma$ is also nonnegative since $\bz$ and $\by$ are both in $D$ and thus $\alpha$ and $\gamma$ are nonnegative. Thus \eqref{eq:NonnegPolynomial1} holds provided
\begin{equation}\label{eq:NonnegPolynomial2}
\xi_1 \alpha^2 - \xi_2 \alpha\beta + \xi_3 \delta^2 \geq 0\,,
\end{equation}
where $\textstyle \xi_1 = \left( \frac{C_{\eta}^2}{\eta^2} - 1 \right)$, $\textstyle \xi_2 = 2 \left( \frac{C_{\eta}^2}{\eta} + 1 \right)$ and $\xi_3 = C_{\eta}^2 - 1$. Since $\beta \leq \rmM \delta$ and $\xi_2$ is nonnegative, \eqref{eq:NonnegPolynomial2} will in turn hold provided
\begin{equation}\label{eq:NonnegPolynomial3}
\Theta(\alpha,\delta) := \xi_1 \alpha^2 - \xi_2 \rmM \alpha\delta + \xi_3 \delta^2 \geq 0\,.
\end{equation}
After completing the square we may rewrite $\Theta(\alpha,\delta) $ as 
\[
\Theta(\alpha,\delta) =\xi_1\left(\alpha-{\xi_2 M\over 2\xi_1}\delta\right)^{2} + \left(\xi_3- {\xi_2^2M^2 \over 4\xi_1}\right)\delta^2. 
\]
Thus, to prove the lemma it is sufficient to have that $\xi_3- {\xi_2^2M^2 \over 4\xi_1}\geq 0$, which is equivalent to the assumption  $\textstyle \mathrm{M}^2 \leq \frac{\left( C_{\eta}^2-\eta^2 \right) \left( C_{\eta}^2-1 \right) }{(C_{\eta}^2 + \eta)^2}$. That concludes the proof of the lemma. 
\end{proof}
\begin{corollary} \label{rmk:UpperBoundOnM} 
There is $\rmM_0>0$ such that for any $\eta>0$,  $\rmM<\rmM_0$, and $\bz$, $\by \in D$, it follows that 
\begin{equation}\label{rmk:UpperBoundOnM2}
|\bz - \by| \leq 2 \max \{ 1, \eta \}\left| (\Phi_{\eta})^{-1}(\bz) - \by \right|. 
\end{equation}

\end{corollary}
\begin{proof}
For a given $\eta>0$, choose $C_{\eta} = 2 \max \{ 1, \eta \}$ in Lemma \ref{lma:GeometricInequality}. Then we have the lower bound $\textstyle {9\over 25} <  \frac{\left( C_{\eta}^2-\eta^2 \right) \left( C_{\eta}^2-1 \right) }{(C_{\eta}^2 + \eta)^2}$ for every $\eta > 0$. Thus, if we take $\rmM_0 ={3\over 5}$, then the assumption $\rmM < {3\over 5}$ is sufficient to prove \eqref{eq:ComparisonOfMeasures} with this choice of $C_{\eta}$. 
\end{proof}
\begin{lemma}\label{lma:StraighteningOfOmega}
Let $\rmM_0>0,$ and $D$ is an epigraph supported by a Lipschitz function $f$ with Lipschitz constant $\rmM<\rmM_0$. Suppose $\bu : D \to \bbR^d$, $p \in (1,\infty)$, $s \in (0,1)$. Define $\bv : \bbR^d_+ \to \bbR^d$ by
\begin{equation*}
\bv(\bx',x_d) = \bu(\bx',f(\bx')+x_d)\,.
\end{equation*}
Then if $\bu \in C^{0}(D)\cap W^{s,p}(D)$, then $\bv \in C^{0}(\bbR^d_+)\cap \cX^{s,p}(\bbR^d_+)$, with
\begin{equation*}
[\bv]_{\cX^{s,p}(\bbR^d_+)} \leq C \left( [\bu]_{\cX^{s,p}(D)} + \rmM [\bu]_{W^{s,p}(D)} \right)\,,
\end{equation*}
where $C$ is independent of $\rmM$ but depends on $\rmM_0$, $d$, $s$, and $p$.
\end{lemma}

\begin{proof}
Define $\Psi : D \to \bbR^d_+$ by
\begin{equation*}
\Psi(\bx) := (\bx',x_d - f(\bx'))\,.
\end{equation*}
Then
\begin{equation*}
\Psi^{-1}(\bx) = (\bx',x_d + f(\bx'))\,, \quad |\grad (\Psi^{-1}(\bx))| \leq d(1+\rmM)\,, \quad \det \grad (\Psi^{-1}(\bx)) = 1\,.
\end{equation*}
With this, 
\begin{equation*}
\begin{split}
[\bv]_{\cX^{s,p}(\bbR^d_+)}^p &= \iintdm{\bbR^d_+}{\bbR^d_+}{\frac{\left| \big( \bu(\Psi^{-1}(\bx) - \bu(\Psi^{-1}(\by) \big) \cdot (\bx-\by) \right|^p}{|\bx-\by|^{d+sp+p}}}{\by}{\bx} \\
	&\leq \iintdm{D}{D}{\frac{\left| \big( \bu(\bx) - \bu(\by) \big) \cdot \big( \Psi(\bx)-\Psi(\by) \big) \right|^p}{|\Psi(\bx)-\Psi(\by)|^{d+sp+p}}}{\by}{\bx} \\
	&\leq C\iintdm{D}{D}{\frac{\left| \big( \bu(\bx) - \bu(\by) \big) \cdot \big( \bx-\by \big) \right|^p}{|\bx-\by|^{d+sp+p}}}{\by}{\bx} \\
	&\qquad + C \iintdm{D}{D}{\frac{\left| \big( u_d(\bx) - u_d(\by) \big) \cdot \big( f(\bx')-f(\by') \big) \right|^p}{|\bx-\by|^{d+sp+p}}}{\by}{\bx} \\
	&\leq C \left(  [\bu]_{\cX^{s,p}(D)}^p + \rmM^p \iintdm{D}{D}{\frac{\left| u_d(\bx) - u_d(\by) \right|^p}{|\bx-\by|^{d+sp}} \cdot \frac{ \left| \bx'-\by' \right|^p}{|\bx-\by|^{p}}}{\by}{\bx} \right) \\
	&\leq C \left( [\bu]_{\cX^{s,p}(D)}^p + \rmM^p [\bu]_{W^{s,p}(D)}^p \right)\,.
\end{split}
\end{equation*}
where $C$ is independent of $\rmM$ but depends on $\rmM_0$, $d$, $s$, and $p$.
\end{proof}

\begin{proof}[Proof of Proposition \ref{ext-operator}]
We define our extension in the spirit of the work of Nitsche \cite{nitsche1981korn} which later was used in \cite{MengeshaTadele2019FKaH} in the case of a half-space. For $\bu = (\bu',u_d) \in \big[ C^1_c(D) \big]^d$, and for constants $\lambda$, $\mu$, $k$, $\ell$, $m$, and $n$, set
\begin{equation*}
[ \mathrm{E}(\bu)(\bx) ]_i :=
\begin{cases}
u_i(\bx)\,, & \qquad \bx \in D\,, \quad i = 1, 2, \ldots d-1, d\,, \\
k \, u_i^{\lambda}(\bx) + \ell \, u_i^{\mu}(\bx)\,, & \qquad \bx \in D_-\,, \quad i = 1, 2, \ldots d-1\,, \\
m \, u_d^{\lambda}(\bx) + n \, u_d^{\mu}(\bx)\,, & \qquad \bx \in D_-\,,
\end{cases}
\end{equation*}
where
\begin{equation*}
\begin{split}
u_j^{\lambda}(\bx) := u_j \big( \bx',f(\bx') + \lambda(f(\bx')-x_d) \big)\,,\\
u_j^{\mu}(\bx) := u_j \big( \bx',f(\bx') + \mu(f(\bx')-x_d) \big)\,.
\end{split}
\end{equation*}
We choose constants $\lambda$, $\mu$, $k$, $\ell$, $m$, $n$, such that
\begin{equation}\label{eq:ReflectionConditions}
\begin{split}
\lambda > 0\,, \quad \mu > 0\,, \qquad k+ \ell = 1 = m+n\,, \qquad \lambda k = -m\,, \quad \mu \ell = -n\,.
\end{split}
\end{equation}
For $\lambda \neq \mu$ these constants are uniquely defined.

Clearly by \eqref{eq:ReflectionConditions}, $\mathrm{E}$ is bounded from $\big[ C^0(D) \big]^d$ to $\big[ C^0(\bbR^d) \big]^d$. We need to show \eqref{eq:ExtensionBound}.
Splitting the integrand,
\begin{equation*}
\begin{split}
[\mathrm{E}(\bu)]_{\cX^{s,p}(\bbR^d)} &= \iintdm{D}{D}{\frac{\left| \big( \mathrm{E}(\bu)(\bx)-\mathrm{E}(\bu)(\by) \big) \cdot \dotbxy \right|^p}{|\bx-\by|^{d+sp}}}{\by}{\bx} + 2 \iintdm{D_-}{D}{\cdots}{\by}{\bx} + \iintdm{D_-}{D_-}{\cdots}{\by}{\bx} \\
&:= \mathrm{I} + 2 \, \mathrm{II} + \mathrm{III}\,.
\end{split}
\end{equation*}
Clearly, $\mathrm{I} = [\bu]_{\cX^{s,p}(D)}^p$. We bound $\mathrm{III}$ next. From the definition of the extension $\mathrm{E}(\bu)$ on $D_{-}$, we see that 
\begin{equation*}
\mathrm{III} \leq 2^{p-1} \mathrm{III}_A + 2^{p-1}\mathrm{III}_B\,,
\end{equation*}
where
\begin{equation*}
\begin{split}
\mathrm{III}_A&= \iintdm{D_-}{D_-}{\frac{\left| k \big( (\bu')^{\lambda}(\bx) - (\bu')^{\lambda}(\by) \big) \cdot (\bx'-\by') +  m \big( u_d^{\lambda}(\bx) - u_d^{\lambda}(\by) \big) \cdot (x_d-y_d)\right|^p}{|\bx-\by|^{d+(s+1)p}}}{\by}{\bx}\,, \\
\mathrm{III}_B &= \iintdm{D_-}{D_-}{\frac{\left| \ell \big( (\bu')^{\mu}(\bx) - (\bu')^{\mu}(\by) \big) \cdot (\bx'-\by') +  n \big( u_d^{\mu}(\bx) - u_d^{\mu}(\by) \big) \cdot (x_d-y_d)\right|^p}{|\bx-\by|^{d+(s+1)p}}}{\by}{\bx}\,.
\end{split}
\end{equation*}
We proceed to bound $\mathrm{III}_A$; the bound for $\mathrm{III}_B$ will follow similarly. Introduce the coordinate change \\ $\bz = \Phi_{\lambda}(\bx)$, $\bw = \Phi_{
\lambda}(\by)$. Then
\begin{equation*}
\mathrm{III}_A = \frac{1}{\lambda^2} \iintdm{D}{D}{\frac{\left| k \big( \bu'(\bz) - \bu'(\bw) \big) \cdot (\bz'-\bw') +  m \big( u_d(\bz) - u_d(\bw) \big) \cdot \left( \big[ (\Phi_{\lambda})^{-1}(\bz) \big]_d - \big[ (\Phi_{\lambda})^{-1}(\bw) \big]_d \right) \right|^p}{|(\Phi_{\lambda})^{-1}(\bz) - (\Phi_{\lambda})^{-1}(\bw)|^{d+(s+1)p}}}{\bw}{\bz}\,.
\end{equation*}
We next write $\left( \big[ (\Phi_{\lambda})^{-1}(\bz) \big]_d - \big[ (\Phi_{\lambda})^{-1}(\bw) \big]_d \right) = -\frac{1}{\lambda}(z_d - w_d) + \frac{\lambda}{1+\lambda} \big( f(\bz')-f(\bw') \big)$ and split $\mathrm{III}_A$ into two integrals:
\begin{equation*}
\begin{split}
\mathrm{III}_A &\leq \frac{2^{p-1}}{\lambda^2} \iintdm{D}{D}{\frac{\left| k \big( \bu'(\bz) - \bu'(\bw) \big) \cdot (\bz'-\bw') -  m \big( u_d(\bz) - u_d(\bw) \big) \cdot \frac{1}{\lambda}(z_d - w_d) \right|^p}{|(\Phi_{\lambda})^{-1}(\bz) - (\Phi_{\lambda})^{-1}(\bw)|^{d+(s+1)p}}}{\bw}{\bz} \\
	&\qquad + \frac{2^{p-1}}{\lambda^2} \iintdm{D}{D}{\frac{\left| m \big( u_d(\bz) - u_d(\bw) \big) \cdot \frac{\lambda}{1+\lambda} \big( f(\bz')-f(\bw') \big) \right|^p}{|(\Phi_{\lambda})^{-1}(\bz) - (\Phi_{\lambda})^{-1}(\bw)|^{d+(s+1)p}}}{\bw}{\bz} \\
&:= \frac{2^{p-1}}{\lambda^2} (\imath) + \frac{2^{p-1}}{\lambda^2} (\imath \imath)\,.
\end{split}
\end{equation*}
For $\rmM_0$ to be determined and any $\rmM<\rmM_0$, using the bound
\begin{equation*}
|\bz - \bw| \leq \Vnorm{\grad \Phi_{\lambda}}_{L^{\infty}(D_-)} \left| (\Phi_{\lambda})^{-1}(\bz) - (\Phi_{\lambda})^{-1}(\bw) \right|
\end{equation*}
and using that $\lambda k = -m$, it follows immediately from Lemma \ref{lma-PropertiesOfCoordinateChange} that $(\imath)$ is majorized by a constant $C$ times $ k^p [\bu]_{\cX^{s,p}(\Omega_+)}^p$, where $C$ independent of $\rmM$ and depends only on $\rmM_0$, $\lambda$, $d$, $s$, and $p$.
As for $(\imath \imath)$, the integral is bounded by the $W^{s,p}$ norm of the last component of $\bu$. Precisely,
\begin{equation*}
\begin{split}
(\imath \imath) &\leq \left( \frac{m \lambda}{1+\lambda} \right)^p \Vnorm{\grad \Phi_{\lambda}}_{L^{\infty}(D_-)}^{d+(s+1)p} \iintdm{D}{D}{\frac{|u_d(\bz)-u_d(\bw)|^p \, |f(\bz')-f(\bw')|^p}{|\bz-\bw|^{d+(s+1)p}}}{\bw}{\bz} \\
&\leq \left( \frac{m \lambda}{1+\lambda} \right)^p \Vnorm{\grad \Phi_{\lambda}}_{L^{\infty}(D_-)}^{d+(s+1)p} \rmM^p \iintdm{D}{D}{\frac{|u_d(\bz)-u_d(\bw)|^p}{|\bz-\bw|^{d+sp}} \cdot \frac{|\bz'-\bw'|^p}{|\bz-\bw|^p}}{\bw}{\bz} \\
&\leq C \,\rmM^p [u_d]_{W^{s,p}(D)}^p \leq C\, \rmM^p [\bu]_{W^{s,p}(D)}^p\,,
\end{split}
\end{equation*}
where $C$ is independent of $M$ but  depends on $\rmM_0$,  $m$, $\lambda$, $p$, $s$, and $d$.
Thus, the desired bound for $\mathrm{III}_A$ is achieved. The bound for $\mathrm{III}_B$ is obtained using the same argument with the identity $\mu \ell = -n$ serving the role of the identity $\lambda k = -m$. That completes bounding $\mathrm{III}$. 

It remains to bound $\mathrm{II}$.  Notice in this case that in the integrand $\bx\in D_-$ and $\by\in D$. By adding and subtracting the quantities
\begin{equation*}
\begin{split}
k \big( u_d^{\lambda}(\bx) - u_d(\by) \big) &\cdot \Big( \big[ \Phi_{\lambda}(\bx) \big]_d - y_d \Big)\,, \\
\ell \big( u_d^{\mu}(\bx) - u_d(\by) \big) &\cdot \Big( \big[ \Phi_{\mu}(\bx) \big]_d - y_d \Big)\,,
\end{split}
\end{equation*}
the integrand in $\mathrm{II}$ can be expressed as 
\begin{equation}\label{eq:TermIII:UnsimplifiedIntegrand}
\begin{split}
&k \big( \bu(\Phi_{\lambda}(\bx)) - \bu(\by) \big) \cdot \big( \Phi_{\lambda}(\bx) - \by \big) + \ell \big( \bu(\Phi_{\mu}(\bx)) - \bu(\by) \big) \cdot \big( \Phi_{\mu}(\bx) - \by \big) \\
&+ (m-k) \big( u_d(\Phi_{\lambda}(\bx)) - u_d(\by) \big) \cdot \Big( \big[ \Phi_{\lambda}(\bx) \big]_d - y_d \Big)
+ (n-\ell) \big( u_d(\Phi_{\mu}(\bx)) - u_d(\by) \big) \cdot \Big( \big[ \Phi_{\mu}(\bx) \big]_d - y_d \Big) \\
&\quad + m \big( u_d(\Phi_{\lambda}(\bx)) - u_d(\by) \big) \cdot \Big( x_d - \big[ \Phi_{\lambda}(\bx) \big]_d \Big)
+ n \big( u_d(\Phi_{\mu}(\bx)) - u_d(\by) \big) \cdot \Big( x_d - \big[ \Phi_{\mu}(\bx) \big]_d \Big)\,.
\end{split}
\end{equation}
Using the relations \eqref{eq:ReflectionConditions} satisfied by $k$, $\ell$, $m$ and $n$, and noting also that $k - m = n - \ell$, from rudimentary algebraic calculations we see that \eqref{eq:TermIII:UnsimplifiedIntegrand} simplifies to
\begin{equation*}
\begin{split}
&k \big( \bu(\Phi_{\lambda}(\bx)) - \bu(\by) \big) \cdot \big( \Phi_{\lambda}(\bx) - \by \big) + \ell \big( \bu(\Phi_{\mu}(\bx)) - \bu(\by) \big) \cdot \big( \Phi_{\mu}(\bx) - \by \big) \\
&\quad + (k-m) \big( y_d - f(\bx') \big) \cdot \big( u_d(\Phi_{\lambda}(\bx)) - u_d(\Phi_{\mu}(\bx)) \big)\,.
\end{split}
\end{equation*}
Therefore,
\begin{align*}
\begin{split}
\mathrm{II} &\leq C \iintdm{D_-}{D}{\frac{|k \big( \bu(\Phi_{\lambda}(\bx)) - \bu(\by) \big) \cdot \big( \Phi_{\lambda}(\bx) - \by \big)|^p}{|\bx-\by|^{d+(s+1)p}}}{\by}{\bx} \\
&\qquad + C \iintdm{D_-}{D}{\frac{|\ell \big( \bu(\Phi_{\mu}(\bx)) - \bu(\by) \big) \cdot \big( \Phi_{\mu}(\bx) - \by \big)|^p}{|\bx-\by|^{d+(s+1)p}}}{\by}{\bx} \\
&\qquad + C \iintdm{D_-}{D}{\frac{|(k-m) \big( y_d - f(\bx') \big) \cdot \big( u_d(\Phi_{\lambda}(\bx)) - u_d(\Phi_{\mu}(\bx)) \big)|^p}{|\bx-\by|^{d+(s+1)p}}}{\by}{\bx} \\
&:= C \Big( \mathrm{II}_1 + \mathrm{II}_2 + \mathrm{II}_3 \Big)\,.
\end{split}
\end{align*}
Making the change of variables $\bz = \Phi_{\lambda}(\bx)$, we get
\begin{equation}
\mathrm{II}_1 = \frac{k^p}{\lambda} \iintdm{D}{D}{\frac{\left| \big( \bu(\bz)-\bu(\by) \big) \cdot \big( \bz - \by \big) \right|^p}{\left| (\Phi_{\lambda})^{-1}(\bz) - \by \right|^{d+(s+1)p}}}{\by}{\bz}\,.
\end{equation}

Now we invoke  Corollary \ref{rmk:UpperBoundOnM} to fix $\rmM_0$. Then for any $\textstyle \rmM < \rmM_0$,  we have $\mathrm{II}_1 \leq C [\bu]_{\cX^{s,p}(D)}$ where $C$ is independent of $M$ but depends on $\rmM_0$. 
$\mathrm{II}_2$ is bounded similarly. Thus, it remains to bound $\mathrm{II}_3$.

Choosing $\eta = 1$ in Lemma \ref{lma-PropertiesOfCoordinateChange} and using the substitution $\bz = \Phi_{1}(\bx)$,
\begin{equation*}
\begin{split}
\mathrm{II}_3 &\leq C \iintdm{D}{D}{\frac{|y_d - f(\bz')|^p \, |u_d \big( \Phi_{\lambda} ((\Phi_{1})^{-1}(\bz)) \big) - u_d \big( \Phi_{\mu} ( (\Phi_1)^{-1}(\bz)) \big) |^p }{|(\Phi_1)^{-1}(\bz)-\by|^{d+(s+1)p}}}{\by}{\bz} \\
&= C \intdm{D}{J(\bz) \left| u_d \big( \bz', f(\bz') + \lambda(z_d-f(\bz')) \big) - u_d \big( \bz', f(\bz') + \mu(z_d-f(\bz')) \big) \right|^p}{\bz}\,,
\end{split}
\end{equation*}
where
\begin{equation*}
J(\bz) := \intdm{D}{\frac{|y_d-f(\bz')|^p}{\left( |\bz'-\by'|^2 + |(y_d - f(\bz')) + (z_d-f(\bz'))|^2 \right)^{\frac{d+(s+1)p}{2}}} }{\by}\,.
\end{equation*}
By Lemma \ref{lma:JBound}, for each $\bz\in D$, $J(\bz)$ can be bounded as $ J(\bz)\leq \frac{C}{|x_d-f(\bx')|^{sp}}$where $C$ is a constant independent of $M$ but depends on $d,p,$ and $d$. As a consequence we have 
\begin{equation*}
\mathrm{II}_3 \leq C  \intdm{D}{\frac{\left| u_d \big( \bz', f(\bz') + \lambda(z_d-f(\bz')) \big) - u_d \big( \bz', f(\bz') + \mu(z_d-f(\bz')) \big) \right|^p}{|z_d-f(\bz')|^{sp}}}{\bz}\,.
\end{equation*}
Making another change of variables $x_d = z_d - f(\bz')$ and writing $\bz'$ as $\bx'$,
\begin{equation*}
\mathrm{II}_3 \leq C \intdm{\bbR^d_+}{\frac{|u_d(\bx',f(\bx')+\lambda x_d) - u_d(\bx',f(\bx')+ \mu x_d) |^p}{x_d^{sp}}}{\bx}\,.
\end{equation*}
Now, define $\bv : \bbR^d_+ \to \bbR^d$ by $\bv(\bx',x_d) := \bu(\bx',f(\bx')+x_d)$; note that $\bv \in C^1_c(\bbR^d_+)$. For any function $\bw = (\bw', w_d) : \bbR^d_+ \to \bbR^d$ and for any $\eta >0$, define the linear map $\bF_{\eta}(\bw)$ by
\begin{equation*}
\bF_{\eta}(\bw)(\bx) := \left( \frac{\bw'(\bx',x_d)}{\eta} , w_d(\bx',\eta x_d) \right)\,, \qquad \bx \in \bbR^d_+\,.
\end{equation*}
We now see that $u_d(\bx',f(\bx')+\lambda x_d) - u_d(\bx',f(\bx')+ \mu x_d) = v_d(\bx',\lambda x_d) - v_d(\bx',\mu x_d)$ is the $d^{\text{th}}$ component of the vector field $\bF_{\lambda}(\bv) - \bF_{\mu}(\bv)$, and that $\bF_{\lambda}(\bv) - \bF_{\mu}(\bv) \in C^1_c(\bbR^d_+)$. Therefore, by \cite[Lemma 4.1]{MengeshaTadele2019FKaH} and by Lemma \ref{lma:StraighteningOfOmega} we have 
\begin{equation*}
\mathrm{II}_3 \leq C \intdm{\bbR^d_+}{\frac{|\bF_{\lambda}(\bv)-\bF_{\mu}(\bv)|^p}{x_d^{sp}}}{\bx} \leq C[\bv]_{\cX^{s,p}(\bbR^d_+)}^p \leq C \left( [\bu]_{\cX^{s,p}(\Omega_+)}^p + \rmM^p [\bu]_{W^{s,p}(\Omega_+)}^p \right)\,, 
\end{equation*}
for a constant $C$ independent of $\rmM$. 
The proof is complete.
\end{proof}

\begin{proof}[Proof of Theorem \ref{thm:KornsGraph}]
Take $\rmM_0$ as given in Proposition \ref{ext-operator}. 
By the Fractional Korn-type Inequality on all of $\bbR^d$ \cite{MengeshaScott2018Korn},
\begin{equation*}
\Vnorm{\bu}_{W^{s,p}(D)} \leq \Vnorm{\mathrm{E}(\bu)}_{W^{s,p}(\bbR^d)} \leq \Vnorm{\mathrm{E}(\bu)}_{\cX^{s,p}(\bbR^d)} \leq \widetilde{C}\left( \Vnorm{\bu}_{\cX^{s,p}(D)} + \rmM\Vnorm{\bu}_{W^{s,p}(D)} \right)\,.
\end{equation*}
where $\widetilde{C}$ is independent of $M$. 
We now  choose $\rmM$ small so that $\rmM < 1 / \widetilde{C}$ to complete the proof. 
\end{proof}
\section{Fractional Korn's inequality for smooth domains}\label{sec:smoothdomains}
In this section we prove the main result of the paper  Theorem \ref{thm:KornsC1}. First we prove a preliminary result. 
\begin{lemma}[Truncation of a function near the boundary]\label{lma:TruncationNearBdy}
Let $\Omega \subset \bbR^d$ be a bounded domain. Let $\bu \in \big[ \cX^{s,p}(\Omega) \big]^d$ and let $\psi \in W^{1,\infty}(\Omega)$. 
Then $\psi \bu \in \cX^{s,p}(\Omega)$ with
\begin{equation}\label{eq:TruncationNearBdy}
[\psi \bu]_{\cX^{s,p}(\Omega)} \leq C \|\psi\|_{W^{1,\infty}}\left( [\bu]_{\cX^{s,p}(\Omega)} + \Vnorm{\bu}_{L^p(\Omega)} \right)\,,
\end{equation}
where $C=C(d,s,p,\Omega)$.
\end{lemma}

\begin{proof} The estimate follows by adding and subtracting $\psi(\by) \bu(\bx)$ as follows: 
\begin{equation*}
\begin{split}
[\psi \bu]_{\cX^{s,p}(\Omega)}^p &\leq \iintdm{\Omega}{\Omega}{|\bu(\bx)|^p \frac{|\psi(\by)-\psi(\bx)|^p}{|\bx-\by|^{d+sp}} }{\by}{\bx} + \iintdm{\Omega}{\Omega}{|\psi(\bx)|^p \frac{\diffqbunorm^p}{|\bx-\by|^{d+sp}} }{\by}{\bx} \\
&\leq \Vnorm{\grad \psi}_{L^{\infty}(\Omega)}^p \intdm{\Omega}{|\bu(\bx)|^p \intdm{\Omega}{|\bx-\by|^{-d-sp-p}}{\by}}{\bx} + \|\psi\|_{L^{\infty}}^{p}[\bu]_{\cX^{s,p}(\Omega)}^p \\
&\leq C\Vnorm{\grad \psi}_{L^{\infty}(\Omega)}^p \Vnorm{\bu}_{L^p(\Omega)}^p \intdm{B_{2R}({\bf 0})}{|\bz|^{-d-sp-p}}{\bz} + \|\psi\|_{L^{\infty}}^{p}[\bu]_{\cX^{s,p}(\Omega)}^p\\
& \leq C  \|\psi\|^{p}_{W^{1,\infty}} \left( [\bu]^{p}_{\cX^{s,p}(\Omega)} + \Vnorm{\bu}^{p}_{L^p(\Omega)} \right) \,,
\end{split}
\end{equation*}
where in the next-to-last inequality $R>0$ is chosen so that $\Omega \Subset B_R({\bf 0})$.
\end{proof}

\begin{lemma}[An extension result]\label{lma:ExtensionNearBdy}
Let $\Omega \subset \bbR^d$ be a bounded domain, and let $\Omega \subset \widetilde{\Omega}$, where $\widetilde{\Omega} \subseteq \bbR^d$ is any domain (bounded or unbounded). Suppose that $\bv \in \cX^{s,p}(\Omega)$, and suppose that there exists $\beta > 0$ such that for every $\by \in  \widetilde{\Omega} \setminus \overline{\Omega}$
\begin{equation*}
\distt( \by, \suppp \bv ) \geq \beta > 0\,.
\end{equation*}
Then the function $\widetilde{\bv} : \widetilde{\Omega} \to \bbR^d$ defined to be the extension of $\bv$ by ${\bf 0}$ on $\widetilde{\Omega} \setminus \Omega$ belongs to $\cX^{s,p}(\widetilde{\Omega})$ with
\begin{equation}\label{eq:ExtensionNearBdy}
[\widetilde{\bv}]_{\cX^{s,p}(\widetilde{\Omega})} \leq C \left( [\bv]_{\cX^{s,p}(\Omega)} + \Vnorm{\bv}_{L^p(\Omega)} \right)\,,
\end{equation}
where $C=C(d,s,p,\Omega)$.
\end{lemma}

\begin{proof}
Define $K := \suppp \bv$. Then
\begin{equation*}
\begin{split}
[\widetilde{\bv}]_{\cX^{s,p}(\widetilde{\Omega})}^p &\leq [\widetilde{\bv}]_{\cX^{s,p}(\Omega)}^p + 2 \iintdm{\Omega}{\widetilde{\Omega} \setminus \Omega}{\frac{|\bv(\bx)|^p}{|\bx-\by|^{d+sp}}}{\by}{\bx} \\
	&= [\widetilde{\bv}]_{\cX^{s,p}(\Omega)}^p + 2 \iintdm{K}{\widetilde{\Omega} \setminus \Omega}{\frac{|\bv(\bx)|^p}{|\bx-\by|^{d+sp}}}{\by}{\bx} \\
	&\leq [\widetilde{\bv}]_{\cX^{s,p}(\Omega)}^p + 2 \intdm{K}{|\bv(\bx)|^p}{\bx} \intdm{\widetilde{\Omega} \setminus \Omega}{\frac{1}{\distt(\by, \p K)^{d+sp}}}{\by} \\
\end{split}
\end{equation*}
Since $K$ is compact we can replace $\distt(\by, \p K)$ with $\distt(\by, K)$. The resulting $\by$-integral is finite since $\beta > 0$ and $d+sp > d$, and therefore \eqref{eq:ExtensionNearBdy} is proved.
\end{proof}

\begin{remark}\label{rmk:LemmasForWsp}
Note that both Lemma \ref{lma:TruncationNearBdy} and Lemma \ref{lma:ExtensionNearBdy} hold when the function space $\cX^{s,p}(\Omega)$ is replaced with $\big[ W^{s,p}(\Omega) \big]^d$.
\end{remark}

\begin{proof}[Proof of Theorem \ref{thm:KornsC1}]
We use a covering argument and the extension operator developed in Proposition \ref{ext-operator}. Choose an open set $\Omega_0$ and open spheres with centers on $\p \Omega$ denoted $\{ B_{r_j}(\by^j) \}_{j=1}^N$ such that $\Omega_0 \Subset \Omega$, and that $\Omega_j := \Omega \cap B_{r_j}(\by^j)$ together with $\Omega_0$ forms a cover of $\Omega$. For $j \geq 1$ define $T_j : B_{r_j}(\by_j) \to \bbR^d$ to be the operator consisting of the translation $\by_j \to {\bf 0}$ and a rotation such that $T_j( \p \Omega \cap B_{r_j}(\by_j) )$ coincides with part of the graph of a $C^1$ function $f_j : \bbR^{d-1} \to \bbR$ with bounded gradient. Set $Q_j = T_j(B_{r_j}(\by_j))$, and also define
\begin{equation*}
\begin{split}
Q_j^+ := \{ \bx \in Q_j \, : \, x_d > f_j(\bx') \}\,, &\qquad
Q_j^- := \{ \bx \in Q_j \, : \, x_d < f_j(\bx') \}\,, \\
K_j^+ := \{ \bx \in \bbR^d \, : \, x_d > f_j(\bx') \}\,, &\qquad
K_j^- := \{ \bx \in \bbR^d \, : \, x_d < f_j(\bx') \}\,.
\end{split}
\end{equation*}
Additionally, we choose $T_j$ so that $T_j(\Omega_j) = Q_j^+$. Note that $T_j$ is a bi-Lipschitz function, with Lipschitz constant depending only on $d$ and $\Omega$.
Since $\Omega$ is a $C^1$ domain we can choose $r_j$ so small that the resulting $f_j$ defining the graph domain has Lipschitz constant $\rmM_j$  that is as small as we wish. Let $\{ \vphi_j \}_{j=1}^N \subset \big[ C^{\infty}_c(\bbR^d) \big]^d$ be a $C^{\infty}$ partition of unity subordinate to the collection $\{ \Omega_j \}_{j=0}^N$, i.e.\ $\suppp(\vphi_j) \subset B_{r_j}(\by_j)$ with $\distt( \by, \suppp(\vphi_j)) > 0$ for every $\by \in \Omega \setminus \overline{\Omega_j} $ and $\textstyle \sum_{j=0}^N \vphi_j \equiv 1$ on $\Omega$. Define $\bu_j := \vphi_j \bu$. 

We consider $\bu_0$ first. Define $\widetilde{\bu}_0 : \bbR^d \to \bbR^d$ by
\begin{equation*}
\widetilde{\bu}_0(\bx) := 
\begin{cases}
\bu_0(\bx)\,, & \bx \in \Omega_0\,, \\
{\bf 0}\,, & \bx \in \bbR^d \setminus \Omega_0\,.
\end{cases}
\end{equation*}
Then by the fractional Korn-type inequality \cite{MengeshaScott2018Korn} proved for $\bbR^{d}$, Lemma \ref{lma:ExtensionNearBdy}, and Lemma \ref{lma:TruncationNearBdy}
\begin{equation}\label{eq:MainProofEstimate1}
\begin{split}
[\bu_0]_{W^{s,p}(\Omega)} \leq [\widetilde{\bu}_0]_{W^{s,p}(\bbR^d)} &\leq C [\widetilde{\bu}_0]_{\cX^{s,p}(\bbR^d)} \\
	&\EquationReference{\eqref{eq:ExtensionNearBdy}}{\leq} C \left( [\bu_0]_{\cX^{s,p}(\Omega_0)} + \Vnorm{\bu_0}_{L^p(\Omega_0)}\right) \\
	&\leq C \left( [\bu_0]_{\cX^{s,p}(\Omega)} + \Vnorm{\bu_0}_{L^p(\Omega)}\right) \EquationReference{\eqref{eq:TruncationNearBdy}}{\leq} C \left( [\bu]_{\cX^{s,p}(\Omega)} + \Vnorm{\bu}_{L^p(\Omega)}\right)\,,
\end{split}
\end{equation}
for some $C = C(d,s,p,\Omega)$. Now fix $j \in \{ 1, 2, \ldots, N \}$. Since $\distt( \by, \suppp(\vphi_j)) > 0$ for every $\by \in \Omega \setminus \overline{\Omega_j}$ and $\suppp(\vphi_j)$ is compact, we can use Lemma \ref{lma:ExtensionNearBdy} and Remark \ref{rmk:LemmasForWsp} to obtain the bound
\begin{equation}\label{eq:BoundaryEstimate1}
[\bu_j]_{W^{s,p}(\Omega)} \leq C \left( [\bu_j]_{W^{s,p}(\Omega_j)} + \Vnorm{\bu_j}_{L^p(\Omega_j)} \right)\,.
\end{equation}
Now since $T_j$ consists of a rotation and a translation, $\grad T_j$ is a constant rotation, with $T_j(\bx) - T_j(\by) = (\grad T_j) (\bx-\by)$. Therefore, writing $R_j := \grad T_j$, define $\bv_j(\bx) := R_j \bu_j(T_j^{-1}(\bx))$. Then $\bv_j \in W^{s,p}(Q_j^+)$, with
\begin{equation}\label{eq:BoundaryEstimate2}
[\bu_j]_{W^{s,p}(\Omega_j)}^p = C \iintdm{Q_j^+}{Q_j^+}{\frac{\big| R_j^{\intercal} \bv_j(\bx) - R_j^{\intercal} \bv_j(\by) \big|^p}{|T_j^{-1}(\bx)-T_j^{-1}(\by)|^{d+sp}}}{\by}{\bx} \leq C \iintdm{Q_j^+}{Q_j^+}{\frac{\big| \bv_j(\bx) - \bv_j(\by) \big|^p}{|\bx-\by|^{d+sp}}}{\by}{\bx}\,,
\end{equation}
since $T_j$ is bi-Lipschitz. Now define $\widetilde{\bv}_j : K_j^+ \to \bbR^d$ by
\begin{equation*}
\widetilde{\bv}_j(\bx) := 
\begin{cases}
\bv_j(\bx)\,, & \bx \in Q_j^+ \,, \\
{\bf 0}\,, & \bx \in K_j^+ \setminus Q_j^+\,.
\end{cases}
\end{equation*}
Then $\widetilde{\bv}_j \in C^1_c(K_j^+)$ and clearly
\begin{equation}\label{eq:BoundaryEstimate3}
[\bv_j]_{W^{s,p}(Q_j^+)} \leq [\widetilde{\bv}_j]_{W^{s,p}(K_j^+)}\,.
\end{equation}
Therefore by Theorem \ref{thm:KornsGraph},
\begin{equation}\label{eq:BoundaryEstimate4}
[\widetilde{\bv}_j]_{W^{s,p}(K_j^+)} \leq C [\widetilde{\bv}_j]_{\cX^{s,p}(K_j^+)}\,,
\end{equation}
where $C = C(d,s,p,\Omega)$.
It is clear that $\distt(\by, \suppp \widetilde{\bv}_j) > 0$ for every $\by \in K_j^+ \setminus \overline{Q}_j^+$. Therefore by Lemma \ref{lma:ExtensionNearBdy}
\begin{equation}\label{eq:BoundaryEstimate5}
[\widetilde{\bv}_j]_{\cX^{s,p}(K_j^+)} \leq C \left( [\bv_j]_{\cX^{s,p}(Q_j^+)} + \Vnorm{\bv_j}_{L^p(Q_j^+)} \right)\,.
\end{equation}
Then by changing coordinates,
\begin{equation}\label{eq:BoundaryEstimate6}
\begin{split}
[\bv_j]_{\cX^{s,p}(Q_j^+)}^p &= C \iintdm{\Omega_j}{\Omega_j}{ \frac{\left| \big( R_j \bu_j(\bx) - R_j \bu_j(\by) \big) \cdot \big( T_j(\bx) - T_j(\by) \big)  \right|^p}{|T_j(\bx) - T_j(\by)|^{d+sp+p}} }{\by}{\bx} \\
	&= C \iintdm{\Omega_j}{\Omega_j}{ \frac{\left| \big( R_j \bu_j(\bx) - R_j \bu_j(\by) \big) \cdot \big( R_j \bx - R_j \by \big)  \right|^p}{|R_j \bx - R_j \by|^{d+sp+p}} }{\by}{\bx} \\
	&\leq C \iintdm{\Omega_j}{\Omega_j}{ \frac{\left|R_j^{\intercal} R_j \big( \bu_j(\bx)-\bu_j(\by) \big) \cdot \big( \bx-\by \big)  \right|^p}{|\bx - \by|^{d+sp+p}} }{\by}{\bx} = C [\bu_j]_{\cX^{s,p}(\Omega_j)}^p\,,
\end{split}
\end{equation}
where $C = C(d,s,p,\Omega)$. By Lemma \ref{lma:TruncationNearBdy} and the remark following it, we obtain 
\begin{equation}\label{eq:BoundaryEstimate7}
[\bu_j]_{\cX^{s,p}(\Omega_j)}^p \leq C \left( [\bu]_{\cX^{s,p}(\Omega_j)}^p + \Vnorm{\bu}_{L^p(\Omega_j)}^p\right) \leq C \left( [\bu]_{\cX^{s,p}(\Omega)}^p + \Vnorm{\bu}_{L^p(\Omega)}^p\right)\,.
\end{equation}
Combining inequalities \eqref{eq:BoundaryEstimate1} through
\eqref{eq:BoundaryEstimate7} brings us to the estimate
\begin{equation}\label{eq:MainProofEstimate2}
[\bu_j]_{W^{s,p}(\Omega)} \leq C \left( [\bu]_{\cX^{s,p}(\Omega)} + \Vnorm{\bu}_{L^p(\Omega)} \right)\,.
\end{equation}
Therefore by \eqref{eq:MainProofEstimate1} and \eqref{eq:MainProofEstimate2}
\begin{equation*}
[\bu]_{W^{s,p}(\Omega)} = \left[ \sum_{j=0}^N \bu_j \right]_{W^{s,p}(\Omega)} \leq \sum_{j=0}^N [\bu_j]_{W^{s,p}(\Omega)} \leq C \left( [\bu]_{\cX^{s,p}(\Omega)} + \Vnorm{\bu}_{L^p(\Omega)} \right)\,,
\end{equation*}
which proves the theorem.
\end{proof}

We characterize the dependence of the constant on the domain. In particular, we look at an example of the simplest of bounded $C^1$ domains: the case $\Omega = B_r(\bx_0)$, for $r> 0$ and $\bx_0 \in \bbR^d$.

\begin{corollary}\label{thm:KornsWithScaling}
There exists a constant $C$ depending on $d$, $s$, $p$ and $B_1({\bf 0})$ such that for any ball $B_r(\bx_0)$ of radius $r>0$ centered at $\bx_0$,  
\begin{equation}
\begin{split}
 \iintdm{B_r(\bx_0)}{B_r(\bx_0)}{\frac{|\bu(\bx) - \bu(\by)|^p}{|\bx-\by|^{d+sp}}}{\by}{\bx}&\leq C\iintdm{B_r(\bx_0)}{B_r(\bx_0)}{\frac{\diffqbunorm^p}{|\bx-\by|^{d+sp}}}{\by}{\bx}  \\
	&\qquad + \frac{C}{r^{sp}} \intdm{B_r(\bx_0)}{|\bu(\bx)|^p}{\bx}
\end{split}
\end{equation}
for every $\bu \in \big[ C^1_c(B_r(\bx_0)) \big]^d$.
\end{corollary}
\begin{proof}
We use a scaling argument. Let $\bu \in \big[ C^1_c(B_r(\bx_0)) \big]^d$. Then the function $\bv(\bx) := \frac{\bu(\bx_0 + r \bx)}{r^s}$ belongs to $\big[ C^1_c(B_1({\bf 0})) \big]^d$, with 
$
r^d [ \bv ]_{\cX^{s,p}(B_1({\bf 0}))}^p = [ \bu ]_{\cX^{s,p}(B_r(\bx_0))}^p\,,$ and $ r^d [ \bv ]_{W^{s,p}(B_1({\bf 0}))}^p = [ \bu ]_{W^{s,p}(B_r(\bx_0))}^p\,,
$
by the natural change of coordinates.
By Theorem \ref{thm:KornsC1} on $B_1({\bf 0})$ and multiplying the resulting inequality through by $r^d$,
\begin{equation}\label{eq:KornBallProof2}
r^d [ \bv ]_{W^{s,p}(B_1({\bf 0}))}^p \leq C r^d [ \bv ]_{\cX^{s,p}(B_1({\bf 0}))}^p + C r^d \Vnorm{\bv}_{L^p(B_1({\bf 0}))}^p\,
\end{equation}
for $C = C(d,s,p,B_1({\bf 0}))$. By changing coordinates, \eqref{eq:KornBallProof2} becomes the desired inequality
\begin{equation*}
[ \bu ]_{W^{s,p}(B_r(\bx_0))}^p \leq C [ \bu ]_{\cX^{s,p}(B_r(\bx_0))}^p + Cr^{-sp} \Vnorm{\bu}_{L^p(B_r(\bx_0))}^p\,.
\end{equation*} That concludes the proof. 
\end{proof}

\section{An application of the fractional Korn's inequality}\label{sec:self-improving}
In this section we prove the high differentiability and higher integrability of solutions to \eqref{eq:PLaplaceSystem}. 
The proof of the theorem follows  the argument presented in \cite{kuusi2015}, which is summarized and explained in concise way in \cite{KuusiTuomo2014AfGl}. In fact, we will only present a proof of one result as the rest is done in \cite{kuusi2015} for $p=2$  and \cite{Scott-Mengesha-preprint} for general $p\geq 2$. The argument relies on a new fractional  Gehring lemma that was first proved in \cite{kuusi2015} for $p=2$. This same fractional Gehring lemma is verified to hold for general $p\geq 2$ in \cite{Scott-Mengesha-preprint} following the same line of proof as in \cite{kuusi2015} in relation to self-improving inequalities for double-phase equations.  For a given ${\bf u}\in \big[W^{s,p}(\bbR^d) \big]^d$, the fractional Gehring lemma \cite[Theorem 2.2]{KuusiTuomo2014AfGl} or \cite[Theorem 6.1]{kuusi2015} is applied to a {\em dual pair} $(\rmU, \nu)$ associated to ${\bf u}$ that satisfies a certain {\em reverse H\"older-type inequality} to 
prove the higher integrability of the function $\rmU$ with respect to the measure $\nu$. The dual pair associated to {\bf u} is defined as  $(\rmU, \nu)$, where  for $\epsilon $ sufficiently small, 
\begin{equation}\label{eq:UDefn}
\rmU(\bx,\by) := \frac{|\bu(\bx)-\bu(\by)|}{|\bx-\by|^{s+\veps}}\,\quad \text{and\,\,\,}\nu(\cB) := \int_\cB \frac{1}{|\bx-\by|^{d-\veps p}} \, \rmd \bx \, \rmd \by\,,
\end{equation}  
for any Lebesgue measurable subset $\cB\subseteq \bbR^{2d}$. 
One notices that for any $\bu \in \big[ L^p(\bbR^d) \big]^d$, for any $s\in(0, 1)$ and $p\in (1, \infty)$
\begin{equation*}
\bu \in \big[W^{s,p}(\bbR^d) \big]^d \qquad \text{ if and only if } \qquad \rmU \in L^p(\bbR^{2d}; \nu)\,.
\end{equation*}
As a consequence, once the fractional Gehring lemma  is applied to prove 
$\rmU \in L^{p + \delta}_{loc}(\bbR^{2d}; \nu)\,,$ for sufficiently small $\delta$ and $\epsilon$,  then for any $B\subseteq\bbR^{d}$, we have $\rmU \in L^{p + \delta}(B\times B; \nu)$. Rewriting the latter in terms of ${\bf u}$ we have that $\int_{B}\int_{B}{|{\bf u}
(\by) - {\bf u}(\bx)|^{p+\delta} \over {|\by-\bx|^{d + (p+\delta)s+ \delta\epsilon} }} \, \rmd \by \, \rmd \bx < \infty  $  which is equivalent to saying that ${\bf u}\in [W^{s+{ \delta\epsilon\over p+\delta},p+\delta}_{loc}(\bbR^{d})]^{d} $
which proves the higher integrability and higher differentiability result of Theorem \ref{thm:MainRegularityResult}. 

We emphasize that the fractional Gehring lemma can be applied to the dual pair $(\rmU, \nu)$  if the dual pair associated with ${\bf u}$ satisfies the reverse H\"older-type inequality. For the particular choice of {\bf u}  which is a solution of \eqref{eq:PLaplaceSystem}, this reverse H\"older-type inequality in turn is a consequence of a Cacciopoli-type inequality which directly uses the nonlocal system of equations. For the strongly coupled nonlinear system of nonlocal equations \eqref{eq:PLaplaceSystem}, this inequality is stated in the following.  

\begin{theorem}\label{thm:Caccioppoli}
Let $p\geq 2$, $0<s<1$, with $sp < n$ and $sp \neq 1$. 
Assume $\bu \in \big[ W^{s,p}(\bbR^d) \big]^d$ is a solution to \eqref{eq:PLaplaceSystem} satisfying \eqref{eq:PLaplaceWeakSoln} corresponding to $\bff\in [L^{p'_\ast}_{loc}(\bbR^{d})]^d$.  For $B = B_r(x_0) \subset \bbR^n$ be a ball, and let $ \psi \in C^{\infty}_c( B)$ such that $0 \leq \psi \leq 1$, $\suppp \psi \subset \frac{1}{2} B$ and $|\grad \psi| \leq \frac{C(d)}{r}$. Then
\begin{equation}\label{eq:CaccioppoliEstimate}
\begin{split}
\iintdm{B}{B}{\frac{| \psi(\bx) \bu(\bx) - \psi(\by)\bu(\by)|^p}{|\bx-\by|^{d+sp}}}{\by}{\bx} & \leq \frac{C}{r^{sp}} \intdm{B}{|\bu(\bx)|^p}{\bx} + C \intdm{\bbR^d \setminus B}{\frac{|\bu(\by)|^{p-1}}{|\bx_0 - \by|^{d+sp}}}{\by} \intdm{B}{|\bu(\bx)|}{\bx} \\
& + Cr^{d+sp'} \left( \fint_B |\bff(\bx)|^{p'_*} \, \rmd \bx \right)^{p'/p'_*}\,.
\end{split}
\end{equation}
for some $C = C(d,s,p,\Lambda) > 0$.
\end{theorem}
Once we prove the above theorem, then deriving the reverse H\"older-type inequality for the dual pair $(\rmU, \nu)$ associated to a solution ${\bf u}$ to \eqref{eq:PLaplaceSystem} can be done in exactly the same way as in \cite{kuusi2015} and \cite{Scott-Mengesha-preprint}. 
The Caccioppoli-type inequality stated in Theorem \ref{thm:Caccioppoli} is therefore the only missing result that one needs to prove Theorem \ref{thm:MainRegularityResult}. Since the inequality relies on the fact that ${\bf u}$ is a solution to the strongly coupled equation \ref{eq:PLaplaceSystem}, the proof of this inequality will -  unlike the proof of the Caccioppoli inequality for scalar nonlocal equations - use the fractional Korn's inequality. In addition the proof needs the following standard result concerning Sobolev spaces. 
\begin{lemma}[Fractional Poincar\'e-Sobolev Inequality]\label{thm:SobolevInequality}
Let $q \in [1,\infty)$, $0<t<1$. Let $B = B_r(\bx_0)$ for some $r>0$, $\bx_0 \in \bbR^d$. Then there exists $C = C(d,s) > 0$ such that 
\begin{equation}
\left( \fint_B \left| \frac{\bv(\bx)}{r^t} \right|^{q^*} \, \rmd \bx \right)^{1/q^*} \leq C \left( \int_B \fint_B \frac{|\bv(\bx)-\bv(\by)|^q}{|\bx-\by|^{d+tq}} \, \rmd \by \, \rmd \bx \right)^{1/q}
\end{equation}
for every $\bv \in \big[ W_0^{t,q}(B) \big]^d$, where  $q^* = {dq\over d-tq }$ is the Sobolev conjugate of $q$. 
\end{lemma}

\begin{proof}[Proof of Theorem \ref{thm:Caccioppoli}]
Since $\big[ C^{\infty}_c(\bbR^d) \big]^d$ is dense in $\big[ W^{s,p}(\bbR^d) \big]^d$ the choice of $\psi^p(\bx) \bu(\bx)$ as the test function in \eqref{eq:PLaplaceWeakSoln} is valid. Testing the system by $\psi^p(\bx) \bu(\bx)$ we have that $\cE_{p,A}(\bu,\psi^p \bu) = \intdm{B}{\psi^p(\bx) \bff(\bx) \cdot \bu(\bx)}{\bx}.$ Writing $\cE_{p,A}(\bu,\psi^p \bu)  = \rmI + \mathrm{II} $ where 
\begin{equation*}
\begin{split}
\rmI&= \iintdm{B}{B}{\frac{A(\bx,\by)}{|\bx-\by|^{d+sp}} |\cD(\bu)(\bx,\by)|^{p-2} \cD(\bu)(\bx,\by) \, \cD(\psi^p \bu)(\bx,\by)}{\by}{\bx} \\
	\mathrm{II}&= 2\iintdm{B}{\bbR^d \setminus B}{\frac{A(\bx,\by)}{|\bx-\by|^{d+sp}} |\cD(\bu)(\bx,\by)|^{p-2} \cD(\bu)(\bx,\by) \, \psi^p(\bx) \left( \bu(\bx) \cdot \dotbxy\right)}{\by}{\bx} \\
\end{split}
\end{equation*}
we will estimate each term separately, then collect terms.

\noindent {\bf  Estimate of $\rmI$.} We assume first that $\psi(\bx) \geq \psi(\by)$. By adding and subtracting $\psi^p(\bx) \bu(\by) \cdot \dotbxy$,
\begin{equation}\label{eq:Caccioppoli:IEstimate1}
\begin{split}
|\cD(\bu)(\bx,\by)|^{p-2} & \cD(\bu)(\bx,\by) \cD(\psi^p \bu)(\bx,\by) \\
	&= \psi^p(\bx) |\cD(\bu)(\bx,\by)|^p
	+ \big( \psi^p(\bx) - \psi^p(\by) \big) | \cD(\bu)(\bx,\by)|^{p-2} \cD(\bu)(\bx,\by) \left( \bu(\by) \cdot \dotbxy \right) \\
	&= \psi^p(\bx) |\cD(\bu)(\bx,\by)|^p + \rmR_1\,.
\end{split}
\end{equation}
We will bound $\rmR_1$ from below. By the assumption $\psi(\bx) \geq \psi(\by)$ we have that for some $\sigma\in [0, 1]$ 
\begin{equation}\label{eq:Caccioppoli:IEstimate2}
\begin{split}
\psi^p(\bx) - \psi^p(\by) &= p \big( \sigma \psi(\bx) + (1-\sigma) \psi(\by) \big)^{p-1} (\psi(\bx) - \psi(\by)) \\
	&\geq - p \big| \sigma \psi(\bx) + (1-\sigma) \psi(\by) \big|^{p-1} |\psi(\bx) - \psi(\by)| \geq - p |\psi(\bx)|^{p-1} |\psi(\bx) - \psi(\by)|\,. 
\end{split}
\end{equation}
Then using \eqref{eq:Caccioppoli:IEstimate2} and Young's Inequality,
\begin{equation}\label{eq:Caccioppoli:IEstimate3}
\begin{split}
\rmR_1 &= p \big( \sigma \psi(\bx) + (1-\sigma) \psi(\by) \big)^{p-1} (\psi(\bx) - \psi(\by)) | \cD(\bu)(\bx,\by)|^{p-2} \cD(\bu)(\bx,\by) \left( \bu(\by) \cdot \dotbxy \right) \\
	&\geq - p | \psi(\bx) |^{p-1} |\psi(\bx) - \psi(\by)| |\cD(\bu)(\bx,\by)|^{p-1} |\bu(\by)|  \\
	&\geq - \frac{1}{p'} \psi^p(\bx) |\cD(\bu)(\bx,\by)|^p - p^{p-1} |\psi(\bx)-\psi(\by)|^p |\bu(\by)|^p
\end{split}
\end{equation}
Combining \eqref{eq:Caccioppoli:IEstimate1} and \eqref{eq:Caccioppoli:IEstimate3} gives
\begin{equation}\label{eq:Caccioppoli:IEstimate4}
|\cD(\bu)(\bx,\by)|^{p-2} \cD(\bu)(\bx,\by) \cD(\psi^p \bu)(\bx,\by) \geq C \psi^p(\bx) |\cD(\bu)(\bx,\by)|^p - C' |\psi(\bx)-\psi(\by)|^p |\bu(\by)|^p
\end{equation}
in the case that $\psi(\bx) \geq \psi(\by)$. Now we assume that $\psi(\by) \geq \psi(\bx)$. By adding and subtracting $\psi^p(\by) \bu(\bx) \cdot \dotbxy$ and proceeding similarly to the first case,
\begin{equation}\label{eq:Caccioppoli:IEstimate5}
|\cD(\bu)(\bx,\by)|^{p-2} \cD(\bu)(\bx,\by) \cD(\psi^p \bu)(\bx,\by) \geq C \psi^p(\by) |\cD(\bu)(\bx,\by)|^p - C' |\psi(\bx)-\psi(\by)|^p |\bu(\bx)|^p\,.
\end{equation}
Using the lower bound on $A$, symmetry, and the estimates \eqref{eq:Caccioppoli:IEstimate4} and \eqref{eq:Caccioppoli:IEstimate5},
\begin{equation}\label{eq:Caccioppoli:IEstimate6}
\rmI \geq C \iintdm{B}{B}{\frac{|\cD(\bu)(\bx,\by)|^p}{|\bx-\by|^{d+sp}} \max \{\psi^p(\bx),\psi^p(\by) \} }{\by}{\bx} - C' \iintdm{B}{B}{\frac{|\psi(\bx) - \psi(\by)|^p}{|\bx-\by|^{d+sp}} |\bu(\bx)|^p }{\by}{\bx} \,,
\end{equation}
where $C = C(p,\Lambda)$. Finally, since
\begin{equation*}
 \left| \big( \psi(\bx) \bu(\bx) - \psi(\by) \bu(\by) \big) \cdot \dotbxy \right|^p \leq 2^{p-1} \psi^p(\by) \left| \big( \bu(\bx) - \bu(\by) \big) \cdot \dotbxy \right|^p + 2^{p-1} |\bu(\bx)|^p |\psi(\bx) - \psi(\by)|^p
\end{equation*}
we obtain
\begin{equation}\label{eq:Caccioppoli:IEstimate7-0}
\rmI \geq C \iintdm{B}{B}{\frac{|\cD(\psi \bu)(\bx,\by)|^p}{|\bx-\by|^{d+sp}} }{\by}{\bx} - C' \iintdm{B}{B}{\frac{|\psi(\bx) - \psi(\by)|^p}{|\bx-\by|^{d+sp}} |\bu(\bx)|^p }{\by}{\bx} \,.
\end{equation}
Now, since $|\grad \psi| \leq Cr^{-1}$ the second integral on the right-hand side of \eqref{eq:Caccioppoli:IEstimate7-0} can be estimates from below by 
\begin{equation}\label{eq:Caccioppoli:Piece1}
-C r^{-p} \int_B |\bu(\bx)|^p \int_{B}  |\bx-\by|^{-d+(1-s)p} \, \rmd \by \, \rmd \bx \leq -C r^{-sp} \int_B |\bu(\bx)|^p \, \rmd \bx\,.
\end{equation}
Therefore we have 
\begin{equation}\label{eq:Caccioppoli:IEstimate7}
\rmI \geq C \iintdm{B}{B}{\frac{|\cD(\psi \bu)(\bx,\by)|^p}{|\bx-\by|^{d+sp}} }{\by}{\bx}-C r^{-sp} \int_B |\bu(\bx)|^p \, \rmd \bx. 
\end{equation}
\noindent{\bf Estimate of $\mathrm{II}$.}
We begin by directly estimating as 
\[
|\cD(\bu)(\bx,\by)|^{p-2} \cD(\bu)(\bx,\by) \psi^p(\bx) \left( \bu(\bx) \cdot \dotbxy \right) \geq - |\cD(\bu)(\bx,\by)|^{p-1} \psi^p(\bx)  |\bu(\bx)|\]
	Since $p\geq 2$, using the inequality $(a+b)^{p-1}\leq 2^{p-2}(a^{p-1} + b^{p-1})$, we have 
	\begin{equation}\label{eq:Caccioppoli:IEstimate8}
	|\cD(\bu)(\bx,\by)|^{p-2} \cD(\bu)(\bx,\by) \psi^p(\bx) \left( \bu(\bx) \cdot \dotbxy \right) \geq -2^{p-2}(|\psi(\bx)\bu(\bx)|^{p} +\psi^p(\bx) |\bu(\bx)| |\bu(\by)|^{p-1})
	\end{equation}
Therefore,
\begin{equation}\label{eq:Caccioppoli:IEstimate9}
\mathrm{II} \geq - \frac{C}{\Lambda} \iintdm{B}{\bbR^d \setminus B}{\psi^p(\bx) \frac{|\bu(\bx)|^{p}  + |\bu(\by)|^{p-1}|\bu(\bx)|}{|\bx-\by|^{d+sp}}}{\by}{\bx}\,.
\end{equation}
we have that 
\begin{equation*}
\frac{|\bx_0-\by|}{|\bx-\by|} \leq \frac{|\bx_0 - \bx| + |\bx-\by|}{|\bx-\by|} = 1 + \frac{|\bx_0 - \bx|}{|\bx-\by|} \leq 2\,.
\end{equation*}
Thus we can replace $|\bx-\by|$ with $|\bx_0 - \by|$ in \eqref{eq:Caccioppoli:IEstimate9} to obtain the inequality 
\begin{equation}\label{eq:Caccioppoli:IEstimate10}
\begin{split}
\mathrm{II} &\geq - C \iintdm{B}{\bbR^d \setminus B}{\psi^p(\bx) \frac{|\bu(\bx)|^{p} + |\bu(\by)|^{p-1}|\bu(\bx)|}{|\bx_0-\by|^{d+sp}}}{\by}{\bx} \\
	&= - {C\over r^{ps}} \intdm{B}{\psi^p(\bx) |\bu(\bx)|^{p} }{\bx} 
	 - C \intdm{B}{\psi^p(\bx) |\bu(\bx)| }{\bx} \intdm{\bbR^d \setminus B}{\frac{|\bu(\by)|^{p-1}}{|\bx_0-\by|^{d+sp}}}{\by}\,.
\end{split}
\end{equation}
where we have used the fact that $0\leq \psi\leq 1$ and $\int_{\bbR^{d}\setminus B}|\bx_0-\bx|^{-d-ps}d\bx = C r^{-ps.}$

Finally we estimates the right hand side $\intdm{B}{\psi^p(\bx) \bff(\bx) \cdot \bu(\bx)}{\bx}$. To that end,  
by H\"older's inequality using the fact that $p^\ast$ and $p'_\ast$ are H\"older conjugates we have  
\begin{equation*}
\begin{split}
\intdm{B}{\psi^p(\bx) \bff(\bx) \cdot \bu(\bx)}{\bx} &\leq \left( \intdm{B}{|\psi(\bx) \bu(\bx)|^{p^*}}{\bx} \right)^{1/p^*} \left( \intdm{B}{|\bff(\bx)|^{p'_*}}{\bx} \right)^{1/p'_*} \\
	& = r^d \left( \fint_B |\psi(\bx) \bu(\bx)|^{p^*} \, \rmd \bx \right)^{1/p^*} \left( \fint_B |\bff(\bx)|^{p'_*} \, \rmd \bx \right)^{1/p'_*}\,.
\end{split}
\end{equation*}
Using the Sobolev-Poincar\'e inequality (Theorem  \ref{thm:SobolevInequality}) on $\psi \bu$, we arrive at the estimate 
\begin{equation*}
\begin{split}
\intdm{B}{\psi^p(\bx) \bff(\bx) \cdot \bu(\bx)}{\bx} &\leq C r^{d/p'+s} \left( \int_B \int_B \frac{|\psi(\bx) \bu(\bx) - \psi(\by) \bu(\by)|^p}{|\bx-\by|^{d+sp}} \, \rmd \by \, \rmd \bx \right)^{1/p} \left( \fint_B |\bff(\bx)|^{p'_*} \, \rmd \bx \right)^{1/p'_*}\,.
\end{split}
\end{equation*}
By Young's inequality with $\sigma \in (0,1)$ suitably small, 
\begin{equation}\label{eq:Caccioppoli:IEstimate12}
\intdm{B}{\psi^p(\bx) \bff(\bx) \cdot \bu(\bx)}{\bx}\leq \frac{C}{\sigma} r^{d+sp'} \left( \fint_B |\bff(\bx)|^{p'_*} \, \rmd \bx \right)^{p'/p'_*} + \sigma \int_B \int_B \frac{|\psi(\bx) \bu(\bx) - \psi(\by) \bu(\by)|^p}{|\bx-\by|^{d+sp}} \, \rmd \by \, \rmd \bx\,.
\end{equation}

Putting together \eqref{eq:Caccioppoli:IEstimate7}, \eqref{eq:Caccioppoli:IEstimate10}, 
and \eqref{eq:Caccioppoli:IEstimate12}, there exists $C = C(d,s,p,\Lambda)$ and an arbitrarily small $\sigma \in (0,1)$ such that
\begin{equation}\label{eq:Caccioppoli:MajorEstimate1}
\begin{split}
\iintdm{B}{B}{\frac{|\cD(\psi \bu)(\bx,\by)|^p}{|\bx-\by|^{d+sp}} }{\by}{\bx} 
&\leq C r^{-ps} \intdm{B}{ |\bu(\bx)|^{p} }{\bx}  + C r^{d+sp'} \left( \fint_B |\bff(\bx)|^{p'_*} \, \rmd \bx \right)^{p'/p'_*} \\
	& +C \intdm{B}{\psi^p(\bx) |\bu(\bx)| }{\bx} \intdm{\bbR^d \setminus B}{\frac{|\bu(\by)|^{p-1}}{|\bx_0-\by|^{d+sp}}}{\by}\\
	&+  \sigma \int_B \int_B \frac{|\psi(\bx) \bu(\bx) - \psi(\by) \bu(\by)|^p}{|\bx-\by|^{d+sp}} \, \rmd \by \, \rmd \bx\,.
\end{split}
\end{equation}
We can now apply fractional  Korn's inequality for balls on $\psi \bu$ Corollary \ref{thm:KornsWithScaling} to obtain 
\begin{equation}\label{eq:Caccioppoli:Piece2}
C  \int_B \int_B \frac{|\psi(\bx) \bu(\bx) - \psi(\by) \bu(\by)|^p}{|\bx-\by|^{d+sp}} \, \rmd \by \, \rmd \bx - r^{-sp} \intdm{B}{ |{\psi} (\bx)\bu(\bx)|^{p} }{\bx} \leq\iintdm{B}{B}{\frac{|\cD(\psi \bu)(\bx,\by)|^p}{|\bx-\by|^{d+sp}} }{\by}{\bx} \,,
\end{equation}
where $C = C(d,s,p)$ does not depend on $r$.
Using \eqref{eq:Caccioppoli:Piece1} and \eqref{eq:Caccioppoli:Piece2} in \eqref{eq:Caccioppoli:MajorEstimate1} gives
\begin{equation}\label{eq:Caccioppoli:MajorEstimate2}
\begin{split}
C \iintdm{B}{B}{\frac{|\psi(\bx) \bu(\bx) - \psi(\by) \bu(\by)|^p}{|\bx-\by|^{d+sp}} }{\by}{\bx} &\leq  C\,r^{-sp} \intdm{B}{|\bu(\bx)|^p}{\bx} +C r^{d+sp'} \left( \fint_B |\bff(\bx)|^{p'_*} \, \rmd \bx \right)^{p'/p'_*} \\
	&+ C \intdm{B}{\psi^p(\bx) |\bu(\bx)| }{\bx} \intdm{\bbR^d \setminus B}{\frac{|\bu(\by)|^{p-1}}{|\bx_0-\by|^{d+sp}}}{\by}\\
	&+  \sigma \int_B \int_B \frac{|\psi(\bx) \bu(\bx) - \psi(\by) \bu(\by)|^p}{|\bx-\by|^{d+sp}} \, \rmd \by \, \rmd \bx\,.
\end{split}
\end{equation}
Since $\sigma \in (0,1)$ can be as small as we wish, we can absorb the last term on the right-hand side of \eqref{eq:Caccioppoli:MajorEstimate2}, which proves the result.
\end{proof}

\appendix
\section{Technical Lemmas}

\begin{lemma}\label{lma:JBound}

Let $\rmM_0>0,$ and let $D$ be an epigraph supported by a Lipschitz function $f$ with Lipschitz constant $\rmM<\rmM_0$. Then for every $\bx \in D$
\begin{equation*}
\begin{split}
J(\bx) &:= \intdm{D}{\frac{|y_d-f(\bx')|^p}{\left( |\bx'-\by'|^2 + |(y_d - f(\bx')) + (x_d-f(\bx'))|^2 \right)^{\frac{d+(s+1)p}{2}}} }{\by} \leq \frac{C}{|x_d-f(\bx')|^{sp}}\,,
\end{split}
\end{equation*}
where $C$ is independent of $\rmM$ but depends on $\rmM_0$, $d$, $s$, and $p$.
\end{lemma}

\begin{proof}
By adding and subtracting $f(\by')$ both in the numerator and denominator, and then making the substitution $z_d = y_d-f(\by')$, we obtain that 
\begin{equation}
\begin{split}
J(\bx) 
&= \intdm{D}{\frac{|y_d-f(\by') + f(\by') - f(\bx')|^p}{\left( |\bx'-\by'|^2 + |(y_d - f(\by')) + (f(\by')-f(\bx')) + (x_d-f(\bx'))|^2 \right)^{\frac{d+(s+1)p}{2}}} }{\by} \\
&= \iintdmt{0}{\infty}{\bbR^{d-1}}{\frac{|z_d + f(\by') - f(\bx')|^p}{\left( |\bx'-\by'|^2 + |z_d + (f(\by')-f(\bx')) + (x_d-f(\bx'))|^2 \right)^{\frac{d+(s+1)p}{2}}} }{\by'}{z_d} \\
&\leq 2^{p-1} (\mathrm{I} + \mathrm{II})\,,
\end{split}
\end{equation}
where
\begin{equation}
\begin{split}
\mathrm{I} &= \iintdmt{0}{\infty}{\bbR^{d-1}}{\frac{|z_d|^p}{\left( |\bx'-\by'|^2 + |z_d + (f(\by')-f(\bx')) + (x_d-f(\bx'))|^2 \right)^{\frac{d+(s+1)p}{2}}} }{\by'}{z_d}\,, \\
\mathrm{II} &= \iintdmt{0}{\infty}{\bbR^{d-1}}{\frac{|f(\by') - f(\bx')|^p}{\left( |\bx'-\by'|^2 + |z_d + (f(\by')-f(\bx')) + (x_d-f(\bx'))|^2 \right)^{\frac{d+(s+1)p}{2}}} }{\by'}{z_d}\,.
\end{split}
\end{equation}
We first bound $\mathrm{I}$. Letting $\bw' = \frac{\bx'-\by'}{|z_d+x_d-f(\bx')|}$ and using the MVT,
\begin{equation*}
\begin{split}
\mathrm{I} &= \intdmt{0}{\infty}{\frac{|z_d|^p}{|z_d+x_d-f(\bx')|^{d+(s+1)p}} \intdm{\bbR^{d-1}}{\frac{1}{\left( \left| \frac{\bx'-\by'}{|z_d+x_d-f(\bx')|} \right|^2 + \left( 1 + \frac{f(\by')-f(\bx')}{|z_d+x_d-f(\bx')|} \right)^2 \right)^{\frac{d+(s+1)p}{2}}}}{\by'} }{z_d} \\
&= \intdmt{0}{\infty}{\frac{|z_d|^p}{|z_d+x_d-f(\bx')|^{1+(s+1)p}} \intdm{\bbR^{d-1}}{\frac{1}{\left( \left| \bw' \right|^2 + \left( 1 - \grad f(\theta) \cdot \bw' \right)^2 \right)^{\frac{d+(s+1)p}{2}}}}{\bw'} }{z_d}\,,
\end{split}
\end{equation*}
where $\theta$ is on the line segment connecting $\bx'$ and $\bx' - |z_d+x_d-f(\bx')| \bw'$. Now,
\begin{equation*}
|1-\grad f(\theta) \cdot \bw'| \geq 1 - |\grad f(\theta) \cdot \bw'| \geq 1 - \rmM |\bw'|\,.
\end{equation*}
Thus,
\begin{equation}\label{eq:JBound1}
\begin{split}
\mathrm{I} &\leq \intdmt{0}{\infty}{\frac{|z_d|^p}{|z_d+x_d-f(\bx')|^{1+(s+1)p}} \intdm{\bbR^{d-1}}{\frac{1}{\left( \left| \bw' \right|^2 + \left( 1 - \rmM |\bw'| \right)^2 \right)^{\frac{d+(s+1)p}{2}}}}{\bw'} }{z_d} \\
&= \intdmt{0}{\infty}{\frac{|z_d|^p}{|z_d+x_d-f(\bx')|^{1+(s+1)p}} \intdm{\bbR^{d-1}}{\frac{1}{\left( (1+\rmM^2)|\bw'|^2 - 2 \rmM |\bw'| + 1 \right)^{\frac{d+(s+1)p}{2}}}}{\bw'} }{z_d} \\
\end{split}\,.
\end{equation}
We now write the second integral in polar coordinates, letting $r = |\bw'|$. Since $\textstyle (1+\rmM^2)r^2 - 2 \rmM r + 1 \geq \max \left\lbrace \frac{1}{1+\rmM^2}\,, \frac{1}{2} r^2 \right\rbrace$,
\begin{equation*}
\begin{split}
\mathrm{I} &\leq \omega_{d-2} \intdmt{0}{\infty}{\frac{|z_d|^p}{|z_d+x_d-f(\bx')|^{1+(s+1)p}} }{z_d} \left( \intdmt{0}{1}{\frac{r^{d-2}}{\left( 1+\rmM^2 \right)^{-\frac{d+(s+1)p}{2}}}}{r} + \intdmt{1}{\infty}{\frac{r^{d-2}}{\left( \frac{r^2}{2} \right)^{\frac{d+(s+1)p}{2}}}}{r} \right) \\
&\leq \omega_{d-2} \left( \frac{(1+\rmM^2)^{\frac{d+(s+1)p}{2}}}{d-1} + \frac{2^{\frac{d+(s+1)p}{2}}}{1+(s+1)p} \right) \intdmt{0}{\infty}{\frac{|z_d|^p}{|z_d+x_d-f(\bx')|^{1+(s+1)p}} }{z_d}\,.
\end{split}
\end{equation*}
Therefore, making the coordinate change $\textstyle a = \frac{z_d}{x_d - f(\bx')}$ in the integral on the previous line, for any $\rmM<\rmM_0$, 
\begin{equation*}
\mathrm{I} \leq C \left( \intdmt{0}{\infty}{\frac{a^p}{|a+1|^{1+sp+p}}}{a} \right) \frac{1}{|x_d - f(\bx')|^{sp}}\,,
\end{equation*}
where $C$ independent of $\rmM$ but  depends on $\rmM_0$, $d$, $s$, and $p$. We have therefore obtained the desired bound for $\mathrm{I}$ since the integral converges absolutely.

The bound for $\mathrm{II}$ follows similarly; the bound analogous to \eqref{eq:JBound1} is
\begin{equation}\label{eq:JBound2}
\mathrm{II} \leq \intdmt{0}{\infty}{\frac{1}{|z_d-x_d-f(\bx')|^{1+sp}}}{z_d} \intdm{\bbR^{d-1}}{\frac{\rmM^p |\bw'|^p}{\left( (1+\rmM^2)|\bw'|^2 - 2 \rmM |\bw'| + 1 \right)^{\frac{d+(s+1)p}{2}}}}{\bw'}\,.
\end{equation}
Using the same lower bound on $(1+\rmM^2)|\bw'|^2 - 2 \rmM |\bw'| + 1$, we proceed just as we did for $\mathrm{I}$; the second integral in \eqref{eq:JBound2} remains finite despite the presence of $|\bw'|^p$ in the numerator. Thus $\mathrm{II} \leq C \rmM^p \,  |x_d-f(\bx')|^{-sp}$, where the constant $C$ is independent of $\rmM$ but depends on $\rmM_0$, $d$, $s$, and $p$.
\end{proof}


\begin{thebibliography}{10}

\bibitem{Auscher2018nonlocal}
Pascal Auscher, Simon Bortz, Moritz Egert, and Olli Saari.
\newblock Nonlocal self-improving properties: a functional analytic approach.
\newblock {\em Tunisian Journal of Mathematics}, 1(2):151--183, 2018.

\bibitem{BBM}
Jean Bourgain, Ha{\"i}m Brezis, and Petru Mironescu.
\newblock Another look at {S}obolev spaces.
\newblock In Jos{\'e}~Luis Menaldi, Edmundo Rofman, and Agnes Sulem, editors,
  {\em Optimal Control and Partial Differential Equations: In Honour of
  Professor Alain Bensoussan's 60th Birthday}, pages 439--455. IOS Press, 2001.

\bibitem{CiarletPhilippeG2010OKi}
Philippe~G Ciarlet.
\newblock On korn’s inequality.
\newblock {\em Chinese Annals of Mathematics, Series B}, 31(5):607--618, 2010.

\bibitem{Du-Navier1}
Qiang Du, Max Gunzburger, R.~B. Lehoucq, and Kun Zhou.
\newblock Analysis of the volume-constrained peridynamic {N}avier equation of
  linear elasticity.
\newblock {\em J. Elasticity}, 113(2):193--217, 2013.

\bibitem{Du-Mengesha-Tian-arxiv}
Qiang Du, Tadele Mengesha, and Xiaochuan Tian.
\newblock Nonlocal criteria for compactness in the space of lp vector fields.

\bibitem{doi:10.1137/1037123}
C.~O. Horgan.
\newblock Korn’s inequalities and their applications in continuum mechanics.
\newblock {\em SIAM Review}, 37(4):491--511, 1995.

\bibitem{KuusiTuomo2014AfGl}
Tuomo Kuusi, Giuseppe Mingione, and Yannick Sire.
\newblock A fractional gehring lemma, with applications to nonlocal equations.
\newblock {\em Rendiconti Lincei - Matematica e Applicazioni}, 25(4):345--358,
  2014.

\bibitem{kuusi2015}
Tuomo Kuusi, Giuseppe Mingione, and Yannick Sire.
\newblock Nonlocal self-improving properties.
\newblock {\em Anal. PDE}, 8(1):57--114, 2015.

\bibitem{MengeshaTadele2019FKaH}
Tadele Mengesha.
\newblock Fractional korn and hardy-type inequalities for vector fields in half
  space.
\newblock {\em Communications in Contemporary Mathematics}, 21(7), 2019.

\bibitem{MengeshaDuElasticity}
Tadele Mengesha and Qiang Du.
\newblock Nonlocal constrained value problems for a linear peridynamic {N}avier
  equation.
\newblock {\em J. Elasticity}, 116(1):27--51, 2014.

\bibitem{mengesha2015VariationalLimit}
Tadele Mengesha and Qiang Du.
\newblock On the variational limit of a class of nonlocal functionals related
  to peridynamics.
\newblock {\em Nonlinearity}, 28(11):3999, 2015.

\bibitem{nitsche1981korn}
Joachim~A Nitsche.
\newblock On {K}orn's second inequality.
\newblock {\em RAIRO. Analyse num{\'e}rique}, 15(3):237--248, 1981.

\bibitem{Ponce2004}
Augusto~C. Ponce.
\newblock An estimate in the spirit of poincaré's inequality.
\newblock {\em Journal of the European Mathematical Society}, 006(1):1--15,
  2004.

\bibitem{MengeshaScott2018Korn}
James Scott and Tadele Mengesha.
\newblock A fractional {K}orn-type inequality.
\newblock {\em Discrete and Continuous Dynamical Systems - A}, 39:3315, 2019.

\bibitem{Scott-Mengesha-preprint}
James~M. Scott and Tadele Mengesha.
\newblock Self-improving inequalities for bounded weak solutions to nonlocal
  double phase equations, https://arxiv.org/abs/2011.11466, 2020. 

\bibitem{Silling2000}
S.~A. Silling.
\newblock Reformulation of elasticity theory for discontinuities and long-range
  forces.
\newblock {\em J. Mech. Phys. Solids}, 48(1):175--209, 2000.

\bibitem{Silling2010}
S.~A. Silling.
\newblock Linearized theory of peridynamic states.
\newblock {\em J. Elasticity}, 99(1):85--111, 2010.

\bibitem{Silling2007}
S.~A. Silling, M.~Epton, O.~Weckner, J.~Xu, and E.~Askari.
\newblock Peridynamic states and constitutive modeling.
\newblock {\em J. Elasticity}, 88(2):151--184, 2007.

\end{thebibliography}

\end{document}